\documentclass[11pt,a4paper]{article}

\usepackage[comma]{natbib}
\usepackage{anysize}
\usepackage{color}
\usepackage{graphicx}
\usepackage{pdfsync}

\usepackage{amsmath,amssymb,amsthm,amsopn}
\renewcommand{\vec}[1]{\boldsymbol{#1}}

\newcommand{\mat}[1]{\mathbf{#1}}

\newtheorem{lemma}{Lemma}
\newtheorem{theorem}{Theorem}
\theoremstyle{remark}

\DeclareMathOperator{\E}{\boldsymbol{\mathbb{E}}}
\DeclareMathOperator{\Var}{Var}

 \DeclareMathOperator{\VEC}{vec}
 \DeclareMathOperator{\tr}{tr}

\def\ISE{\mathrm{ISE}}
\def\MISE{\mathrm{MISE}}

\def\AMISE{\mathrm{AMISE}}

\def\SCV{\mathrm{SCV}}

\def\PI{\mathrm{PI}}
\def\CV{\mathrm{CV}}

\renewcommand{\today}{\begingroup
\number \day\space  \ifcase \month \or January\or February\or
March\or April\or May\or June\or July\or August\or September\or
October\or November\or December\fi \space  \number \year \endgroup}

\newcommand{\D}{\mathsf{D}}
\newcommand{\bH}{{\bf H}}
\newcommand{\bG}{{\bf G}}
\newcommand{\bI}{{\bf I}}
\newcommand{\bw}{{\boldsymbol w}}
\newcommand{\bx}{{\boldsymbol x}}
\newcommand{\by}{{\boldsymbol y}}
\newcommand{\bz}{{\boldsymbol z}}
\newcommand{\bX}{{\bf X}}
\newcommand{\bY}{{\bf Y}}

\newcommand{\bmu}{{\boldsymbol \mu}}
\newcommand{\bpsi}{{\boldsymbol \psi}}

\renewcommand{\S}{\boldsymbol{\mathcal S}}
\newcommand{\bSigma}{{\boldsymbol \Sigma}}

\renewcommand{\vec}{\operatorname{vec}}
\newcommand{\vect}{\operatorname{vec}^\top}

\theoremstyle{plain}

\newtheorem{teor*}{Teorema}

\theoremstyle{definition}

\title{Data-driven density derivative estimation, with applications to nonparametric
clustering and bump hunting}

\author{Jos\'e E. Chac\'on\footnote{Departamento de
Matem\'aticas, Universidad de Extremadura, E-06006 Badajoz, Spain. E-mail:
{\tt jechacon@unex.es}} \ and Tarn Duong\footnote{Theoretical and Applied Statistics Laboratory (LSTA), University Pierre and Marie Curie -- Paris 6, F-75005, Paris, France; Institute of Translational Neurosciences (IHU-A-ICM), Piti\'e-Salp\^etri\`ere Hospital, F-75005, Paris, France. Email:
{\tt tarn.duong@upmc.fr}}}

\begin{document}

\maketitle

\begin{abstract}
\noindent Important information concerning a multivariate data set, such as clusters and modal
regions, is contained in the derivatives of the probability density function. Despite this
importance, nonparametric estimation of higher order derivatives of the density functions have
received only relatively scant attention. Kernel estimators of density functions are widely used as
they exhibit excellent theoretical and practical properties, though their generalization to density
derivatives has progressed more slowly due to the mathematical intractabilities encountered in the
crucial problem of bandwidth (or smoothing parameter) selection.  This paper presents the first
fully automatic, data-based bandwidth selectors for multivariate kernel density derivative
estimators. This is achieved by synthesizing recent advances in matrix analytic theory which allow
mathematically and computationally tractable representations of higher order derivatives of
multivariate vector valued functions. The theoretical asymptotic properties as well as the finite
sample behaviour of the proposed selectors are studied. {In addition, we explore in detail the
applications of the new data-driven methods for two other statistical problems: clustering and bump
hunting. The introduced techniques are combined with the mean shift algorithm to develop novel
automatic, nonparametric clustering procedures which are shown to outperform mixture-model cluster
analysis and other recent nonparametric approaches in practice. Furthermore, the advantage of the
use of smoothing parameters designed for density derivative estimation for feature significance
analysis for bump hunting is illustrated with a real data example.}
\end{abstract}


\newpage

\section{Introduction}

The estimation of density derivatives has full potential for applications. This has been noted even
in the first seminal papers on density estimation, as \cite{Pa62}, which was also concerned with
the estimation of the mode of a unimodal distribution, the value that makes zero the first density
derivative. In the multivariate case, the pioneering work of \cite{FH75} showed how the estimation
of the gradient vector can also be used for clustering and data filtering, leading to a substantial
amount of literature on the subject, under the name of the {\em mean shift algorithm}. Looking
further afield, \cite{Ch95} made use of the mean shift idea for image analysis, and the
highly-cited paper by \cite{CM02} showed how these techniques can be useful for low-level vision
problems, discontinuity preserving smoothing and image segmentation. A further very popular use of
the mean shift algorithm is for real-time object tracking, as described in \cite{CRM03}.

From the perspective of statistical data analysis, in the multidimensional context the estimation
of the first and second derivatives of the density is crucial to identify significant features of
the distribution, such as local extrema, valleys, ridges or saddle points. In this sense,
\cite{GMC02} developed methods for determining and visualizing such features in dimension two,
extending previous work on scale space ideas introduced in \cite{CM99} for the univariate case (the
SiZer approach), and the same authors also explored the application of this methodology to digital
image analysis in \cite{GMC04}. \cite{DCKW08} generalized these results for multivariate data in
arbitrary dimensions and provided a novel visualization for three-dimensional data. These
techniques have been widely applied recently in the field of flow cytometry; see \cite{Ze07},
\cite{NW09} or \cite{Pr09}.

Another relatively new problem that is closely related to gradient estimation is that of finding filaments in point clouds, which has applications in medical imaging, remote sensing, seismology and cosmology. This problem is rigorously stated and analyzed in \cite{GPVW09}. Filaments are one-dimensional curves embedded in a point process, and it can be shown that steepest ascent paths (i.e., the paths from each point following the gradient direction) concentrate around them, so gradient estimation appears as a useful tool for filament detection.

In this paper we focus on kernel estimators of multivariate density derivatives of arbitrary order,
formally defined in Section \ref{sec:2} below. As for any kernel estimator, it is well known that
the crucial factor that determines the performance of the estimator in practice is the choice of
the bandwidth matrix. In the multivariate setting there are several levels of sophistication at the
time of specifying the bandwidth matrix to be used in the kernel estimator \citep[see][Chapter
4]{WJ95}. The most general bandwidth type consists of a symmetric positive definite matrix; it
allows the kernel estimator to smooth in any direction whether coordinate or not. This general
class of bandwidth matrices can be constrained to consider positive definite diagonal matrices,
allowing for different degrees of smoothing along each of the coordinate axis, or even further to
consider a bandwidth matrix involving only a positive scalar multiple of the identity matrix,
meaning that the same smoothing is applied to every coordinate direction. As noted in \cite{WJ93}
in the density estimation context, the single-parameter class should not be used for unscaled data
or, as stated by \cite{CM02} in terms of feature space analysis, to use this bandwidth class at
least the validity of an Euclidean metric for the feature space should be previously checked.

The simpler parameterizations are in general more widely used than the unconstrained counterpart
for two reasons: first, in practice they need less smoothing parameters to be tuned, and second,
due to the difficulties encountered in the mathematical analysis of estimators with unconstrained
bandwidths. However, \cite{CDW11} provided a detailed error analysis of kernel density derivative
estimators with unconstrained bandwidths and showed that the use of the simpler parameterizations
can lead to a substantial loss in terms of efficiency, and that this problem becomes more and more
important as the order of the derivative to be estimated increases.

\cite{CDW11} also proposed an optimal bandwidth selector for the normal case, but they did not
develop more sophisticated data-driven choices of the bandwidth matrix with applicability to more
general densities, which is crucial to make density derivative estimation useful in practice. Along
the same lines, \cite{CM02} argue that most existing bandwidth selection methods for the mean shift
algorithm, all of them for the single-parameter class of bandwidths, are based on empirical
arguments.

In the univariate case there exist some approaches to bandwidth selection for density derivative
estimation: \cite{HMW90} introduced a cross validation method and showed its optimality;
\cite{Jon92} derived the relative rate of convergence of this method and also for a plug-in
proposal; \cite{Wu97} studied two root $n$ selectors in the Fourier domain, and more recently
\cite{DR10} focused on the smoothed cross validation bandwidth selector for the density derivative.
In the multivariate case, however, the issue of automatic bandwidth selection for density
derivative estimation has remained largely unexplored. Given the smaller body of multivariate
density estimation research as compared to their univariate cousins, it is not surprising that
multivariate density derivative estimation suffers equally (if not more so) from a lack of solid
results. To the best of our knowledge, in the literature the only published approaches to bandwidth
selection for multivariate kernel estimation of density derivatives are the recent papers
\cite{HV11} and \cite{HKV13}, but both focus exclusively on the first derivative.

{This paper proposes three new methods for unconstrained bandwidth matrix selection for the
multivariate kernel density derivative estimator, and explores their applicability to other related
statistical problems}. In Section \ref{sec:2}, we introduce the mathematical framework for the
analysis of multivariate derivatives. In Section \ref{sec:3} we show that the relative rate of
convergence of these unconstrained selectors is the same as for the classes of simpler bandwidth
matrices, so that from an asymptotic point of view our methods can be as successful as (and more
flexible than) those needing less smoothing parameters. The finite-sample behaviour of the new
bandwidths is investigated in Section \ref{sec:5}, and their application to develop new data-driven
nonparametric clustering methods via the mean shift algorithm is explored in Section \ref{sec:ms},
and for feature significance in Section \ref{sec:feature}. Finally, the proofs of the results are
given in an appendix.

\section{Kernel density derivative estimation}\label{sec:2}

The problem of estimating the $r$-th derivative of a multivariate density is considered in this
section. From a multivariate point of view, the $r$-th derivative of a function is understood as
the set of all its partial derivatives of order $r$, rather than just one of them. Notice that, for
instance, in a multivariate Taylor expansion of order $r$ all of the partial derivatives of order
$r$ are needed to compute the $r$-th order term. Or, in another related example, all the second
order partial derivatives are involved in the computation of the Hessian matrix.

All the $r$-th partial derivatives can be neatly organized into a single vector as follows: if $f$
is a real $d$-variate density function and $\bx=(x_1,\dots,x_d)$, denote by
$\D=\partial/\partial\bx=(\partial/\partial x_1,\dots,\partial/\partial x_d)$ the first derivative
(gradient) operator. All the second order partial derivatives can be organized into the Hessian
matrix $\mathsf Hf=(\partial^2f/\partial x_i\partial x_j)_{i,j=1}^d$, and the Hessian operator can
be formally written as $\mathsf H=\D\D^\top$ if the usual convention $(\partial/\partial
x_i)(\partial/\partial x_j) = \partial^2/( \partial x_i\partial x_j)$ is taken into account. For
$r\geq3$, however, it is not that clear how to organize the set containing all the $d^r$ partial
derivatives of order $r$. Here we adopt the unified approach used in \cite{Hol96a} or
\citet[Section 1.4]{KvR05}, where the $r$-th derivative of $f$ is defined to be the vector
$\D^{\otimes r} f= (\D f)^{\otimes r}=\partial^r f/\partial{\bx}^{\otimes r}\in\mathbb R^{d^r}$. In
the previous equation $\D^{\otimes r}$ denotes the $r$-th Kronecker power of the operator $\D$;
see, e.g., \cite{MN99} for the definition of the Kronecker product. Naturally, $\D^{\otimes 0}
f=f$, $\D^{\otimes 1}f=\D f$ and, for example, $\D^{\otimes 2}=\vec \mathsf H$, where $\vec$
denotes the operator which concatenates the columns of a matrix into a single vector.

{Here we study the problem of} estimating the $r$-th derivative $\D^{\otimes r}f$ from a sample
$\bX_1,\dots,\bX_n$ of independent and identically distributed random variables with common density
$f$. The usual kernel estimator of $f$ is defined as $\hat
f_{\bH}(\bx)=n^{-1}\sum_{i=1}^nK_\bH(\bx-\bX_i)$, where the kernel $K$ is a spherically symmetric
density function, the bandwidth $\bH$ is a symmetric positive definite matrix and
$K_{\bH}(\bx)=|\bH|^{-1/2}K(\bH^{-1/2}\bx)$. Thus, the most straightforward estimator of
$\D^{\otimes r}f$ is just the $r$-th derivative of $\hat f_{\bH}$, given by $\D^{\otimes r}\hat
f_{\bH}(\bx) =n^{-1}\sum_{i=1}^n\D^{\otimes r}K_\bH(\bx-\bX_i)$, where the roles of $K$ and $\bH$
can be separated for implementation purposes by noting that $\D^{\otimes
r}K_{\bH}(\bx)=|\bH|^{-1/2}(\bH^{-1/2})^{\otimes r}\D^{\otimes r}K(\bH^{-1/2}\bx)$, as shown in
\cite{CDW11}, where for any matrix $\mat A$ it is understood that $\mat A^{\otimes 0}=1\in\mathbb
R$ and $\mat A^{\otimes 1}=\mat A$. See, however, \cite{Jon94} for other possible estimators in the
univariate context.

For the density estimation case ($r=0$), \cite{WJ93} showed that the use of bandwidths belonging to
the class $\mathcal{I} = \lbrace h^2 \bI_d : h \in \mathbb{R} \rbrace$, with $\bI_d$ the $d\times
d$ identity matrix, or the class $\mathcal{D} = \lbrace \mathrm{diag}(h_1^2, h_2^2, \dots, h_d^2) :
h_1, h_2, \dots, h_d \in \mathbb{R} \rbrace$, may lead to dramatically less efficient estimators
than those based on bandwidth matrices drawn from $\mathcal{F}$, the space of all positive definite
symmetric matrices. Moreover \cite{CDW11} showed that the issue of efficiency loss is even more
severe for $r\geq1$. So the development of unconstrained bandwidth selectors for density derivative
estimation, which is {achieved in} this paper, may also represent an important improvement in
practice.

To measure the error committed by the kernel estimator for the sample at hand it is natural to
consider the integrated squared error (ISE), defined as
\begin{align*}
\ISE_r(\bH)&=\int_{\mathbb R^d}\|\D^{\otimes r}\hat f_{\bH}(\bx)-\D^{\otimes r}f(\bx)\|^2d\bx,
\end{align*}
where $\|\cdot\|$ denotes the usual Euclidean norm. This quantity depends on the data, so it is
common to consider the mean integrated squared error $\MISE_r(\bH)=\mathbb E[\ISE_r(\bH)]$ as a
non-stochastic measure of error, and its minimizer $\bH_{\MISE,r}=\mathrm{argmin}_{\bH\in\mathcal F}\MISE_r(\bH)$ as the non-stochastic optimal bandwidth choice. A more detailed discussion of the advantages and disadvantages of the ISE and MISE viewpoints can be found in \cite{Jon91}.

Standard calculations lead to the integrated variance plus integrated squared bias decomposition
$\MISE_r(\bH)={\rm IV}_r(\bH)+{\rm ISB}_r(\bH)$, where ${\rm IV}_r(\bH)=\int_{\mathbb
R^d}\tr\Var[\D^{\otimes r}\hat f_{\bH}(\bx)]d\bx$ and ${\rm ISB}_r(\bH)=\int_{\mathbb R^d}\|\mathbb
E[\D^{\otimes r}\hat f_{\bH}(\bx)]-\D^{\otimes r}f(\bx)\|^2d\bx$. By expanding each of these two
terms, \cite{CDW11} showed that a more explicit form of the MISE is given by
\begin{equation}\label{exactMISE}
\begin{aligned}
\MISE_r(\bH) &= \{n^{-1}|\bH|^{-1/2}\tr\big((\bH^{-1})^{\otimes r}\mathbf R(\D^{\otimes r}K)\big)- n^{-1}\tr \mat{R}^*(K_\bH*K_\bH, \D^{\otimes r} f)\}\\
&\quad +  \{\tr \mat{R}^*(K_\bH*K_\bH, \D^{\otimes r} f) - 2\tr\mat{R}^*(K_\bH, \D^{\otimes r} f)
+ \tr \mat{R}(\D^{\otimes r} f)\}
\end{aligned}
\end{equation}
where $\mat{R}(\mat{g}) = \int_{\mathbb{R}^d} \mat{g}(\bx) \mat{g}(\bx)^\top \,d\bx$, and
$\mat{R}^*(a, \mat{g}) = \int_{\mathbb{R}^d} (a*\mat{g})(\bx) \mat{g}(\bx)^\top \,d\bx$ for a
vector valued function $\mat{g}$ and a real valued function $a$, with $a*\mat{g}$ standing for a
component-wise application of the convolution operator.

Moreover, writing $m_2(K)\bI_d=\int_{\mathbb R^d}\bx\bx^\top K(\bx)d\bx$, under some
smoothness assumptions \citet{CDW11} also showed that an asymptotic approximation of the MISE is
given by
\begin{equation}\label{AMISE}
\begin{aligned}
\mathrm{AMISE}_r(\bH)&=n^{-1}|\bH|^{-1/2}\tr\big((\bH^{-1})^{\otimes r}\mathbf R(\D^{\otimes r}K)\big)\\
&\quad+\tfrac{m_2(K)^2}{4}\tr\big((\bI_{d^r}\otimes\vect\bH)\mathbf
R(\D^{\otimes(r+2)}f)(\bI_{d^r}\otimes\vec\bH)\big)
\end{aligned}
\end{equation}
and that the minimizer of this AMISE function, $\bH_{\AMISE,r} = \mathrm{argmin}_{\bH \in
\mathcal{F}} \AMISE_r(\bH)$, has entries of order $O(n^{-2/(d+2r+4)})$, leading to a minimum
achievable AMISE of order $O(n^{-4/(d+2r+4)})$.

Although these expressions provide an insightful error analysis of multivariate kernel density
derivative estimators, they are not directly implementable as software since they all involve the
unknown density $f$. In the next section, we examine strategies to estimate these unknown
quantities which lead to optimal data-based selectors for density derivative estimation.

\section{Bandwidth selection methods}\label{sec:3}

{In this section we propose three new methods to select the bandwidth matrix for kernel density
derivative estimation from the data and study their asymptotic properties. These methods are
inspired by the cross validation, plug-in and smoothed cross validation methodologies for the
estimation of the density in the univariate case, hence their names.}

\subsection{Cross validation method}
\label{sec:cv}

Cross validation (CV) techniques for bandwidth selection for univariate density estimation were
introduced in \cite{Ru82} and \cite{Bo84}, and studied in detail in seminal papers like
\cite{Ha83}, \cite{St84} and \cite{HM87}. They can be motivated in terms of either ISE-oriented or
MISE-oriented considerations.

In the case of multivariate density derivative estimation notice that, for a random variable $\bX$
having density $f$ and independent of $\bX_1,\dots,\bX_n$, using integration by parts it is
possible to write
\begin{align*}
\ISE_r(\bH)&=(-1)^r\vec^\top\bI_{d^r}\bigg\{n^{-2}\sum_{i,j=1}^n\D^{\otimes2r}K_\bH*K_\bH(\bX_i-\bX_j)-2\mathbb E\big[\D^{\otimes 2r}\hat f_{\bH}(\bX)\big]\bigg\}\\
&\quad+\tr\mat R(\D^{\otimes r}f)
\end{align*}
provided that the kernel $K$ is sufficiently smooth. The last term is irrelevant for minimizing
concerns, and the two first terms can be unbiasedly estimated by
\begin{align*}
\CV_r(\bH)&=(-1)^r\vec^\top\bI_{d^r}\bigg\{n^{-2}\sum_{i,j=1}^n\D^{\otimes2r}K_\bH*K_\bH(\bX_i-\bX_j)-2n^{-1}\sum_{i=1}^n\D^{\otimes 2r}\hat f_{\bH,i}(\bX_i)\bigg\}\\
&\hspace{-1.5cm}=(-1)^r\vec^\top\bI_{d^r}\bigg\{n^{-2}\sum_{i,j=1}^n\D^{\otimes2r}K_\bH*K_\bH(\bX_i-\bX_j)-2[n(n-1)]^{-1}\sum_{i\neq j}\D^{\otimes 2r}K_\bH(\bX_i-\bX_j)\bigg\}
\end{align*}
where $\D^{\otimes 2r}\hat f_{\bH,i}$ denotes the kernel estimator based on the sample
with the $i$-th observation deleted.

From the MISE point of view notice that, for any smooth enough function $L$,
\begin{align*}
\tr\mat R^*(L_\bH, \D^{\otimes r}f) &= \int_{\mathbb R^d}(L_\bH*\D^{\otimes r}f)(\bx)^\top\D^{\otimes r}f(\bx)d\bx \\
&=(-1)^r\vec^\top\bI_{d^r}\int_{\mathbb R^d}\D^{\otimes 2r}L_\bH*f(\bx)f(\bx)d\bx\\
&=(-1)^r\vec^\top\bI_{d^r}\mathbb E\big[\D^{\otimes 2r}L_\bH(\bX_1-\bX_2)\big],
\end{align*}
so that $\tr\mat R^*(L_\bH, \D^{\otimes r}f)$ can be unbiasedly estimated by
$$(-1)^r[n(n-1)]^{-1}\vec^\top\bI_{d^r}\sum_{i\neq j}\D^{\otimes 2r}L_\bH(\bX_i-\bX_j).$$
Therefore, in view of (\ref{exactMISE}), $\MISE_r(\bH)-\tr\mat R(\D^{\otimes r}f)$ can be
unbiasedly estimated by
\begin{align*}
\CV_r(\bH)&=n^{-1}|\bH|^{-1/2}\tr\big((\bH^{-1})^{\otimes r}\mat
R(\D^{\otimes r}K)\big)\\&\quad+(-1)^r[n(n-1)]^{-1}\vec^\top\bI_{d^r}\sum_{i\neq j}^n\Big\{(1-n^{-1})\D^{\otimes 2r}\bar{K}_\bH-2\D^{\otimes 2r}K_\bH\Big\}(\bX_i-\bX_j),
\end{align*}
where $\bar{K}=K*K$. To check that these two CV expressions coincide, take into account that
$\D^{\otimes 2r}K_\bH*K_\bH=\D^{\otimes 2r}(K_\bH*K_\bH)=\D^{\otimes 2r}(K*K)_\bH = \D^{\otimes 2r}
\bar{K}_\bH$, so that using some properties of the Kronecker product and the vec operator
\citep[see, e.g.,][]{MN99}, the sum of the diagonal terms in the first $\CV_r(\bH)$ expression
equals
\begin{align*}
(-1)^rn^{-1}\vec^\top\bI_{d^r}\D^{\otimes2r}\bar{K}_\bH(0)&=(-1)^rn^{-1}|\bH|^{-1/2}\vec^\top\bI_{d^r}(\bH^{-1/2})^{\otimes 2r}\D^{\otimes2r}
\bar{K}(0)\\
&=(-1)^rn^{-1}|\bH|^{-1/2}\vec^\top(\bH^{-1})^{\otimes r}\D^{\otimes2r}\bar{K}(0)\\
&=n^{-1}|\bH|^{-1/2}\vec^\top(\bH^{-1})^{\otimes r}\vec\mat R(\D^{\otimes r}K)\\
&=n^{-1}|\bH|^{-1/2}\tr\big((\bH^{-1})^{\otimes r}\mat
R(\D^{\otimes r}K)\big)
\end{align*}
where the third line follows by noting that $\vec\mat R(\D^{\otimes r}K)=(-1)^r\int_{\mathbb
R^d}\D^{\otimes 2r}K(\bx)K(\bx)d\bx=(-1)^r\D^{\otimes 2r}\bar{K}(0)$.

{Surely} the simplest formulation for CV (useful for implementation purposes) is
\begin{align*}
\CV_r(\bH)&=(-1)^r|\bH|^{-1/2}\vec^\top(\bH^{-1})^{\otimes r}\bigg\{n^{-2}\sum_{i,j=1}^n\D^{\otimes2r}\bar{K}\big(\bH^{-1/2}(\bX_i-\bX_j)\big)\\
&\quad-2[n(n-1)]^{-1}\sum_{i\neq j}\D^{\otimes 2r}K\big(\bH^{-1/2}(\bX_i-\bX_j)\big)\bigg\}.
\end{align*}
We denote by $\hat \bH_{\CV,r}$ the bandwidth matrix in $\mathcal F$ which minimizes ${\rm
CV}_r(\bH)$.

\subsection{Plug-in method}
\label{sec:plugin}

Plug-in (PI) bandwidth selection techniques are based on estimating the unknown quantities that
appear in an asymptotic error formula and minimizing the resulting empirical criterion. Basic
plug-in selectors for univariate density estimation are described in \cite{PM90} and \cite{SJ91}.
In the multivariate case, introducing the vector integrated density derivative functional, defined
as
\begin{align*}
\bpsi_{2r} = \int \D^{\otimes 2r} f(\bx) f(\bx) \, d\bx
&= (-1)^r \vec \mat{R} (\D^{\otimes r} f),
\end{align*}
allows us to rewrite the AMISE formula (\ref{AMISE}) for the $r$-th derivative as
\begin{align*}
\mathrm{AMISE}_r(\bH)&=n^{-1}|\bH|^{-1/2}\tr\big((\bH^{-1})^{\otimes r}\mathbf R(\D^{\otimes r}K)\big) \\
&\quad +(-1)^r \tfrac{m_2(K)^2}{4}\boldsymbol{\psi}_{2r+4}^\top\big(\vec\bI_{d^r}\otimes (\vec\bH)^{\otimes 2}\big).
\end{align*}
Thus, the plug-in bandwidth selector $\hat\bH_{\PI,r}$ is defined to be the bandwidth in $\mathcal
F$ minimizing the criterion
\begin{align*}
\mathrm{PI}_r(\bH)&=n^{-1}|\bH|^{-1/2}\tr\big((\bH^{-1})^{\otimes r}\mathbf R(\D^{\otimes r}K)\big) \\
&\quad + (-1)^r \tfrac{m_2(K)^2}{4} \hat\bpsi{}_{2r+4}^\top\big(\vec\bI_{d^r}\otimes (\vec\bH)^{\otimes 2}\big),
\end{align*}
where $\hat\bpsi_{2r+4}$ is a suitable estimator of $\bpsi_{2r+4}$.

\cite{CD10} analyzed the problem of estimating $\bpsi_{2r}$ for an arbitrary $r$. Noting that $\bpsi_{2r} = \E [ \D^{\otimes 2r} f(\bX)]$, they
proposed the kernel estimator
$$
\hat{\bpsi}_{2r}(\bG) = n^{-1} \sum_{i=1}^n \D^{\otimes 2r} \hat{f}_{n\bG} (\bX_i)
= n^{-2} \sum_{i,j=1}^n \D^{\otimes 2r} L_\bG (\bX_i - \bX_j),
$$
using a kernel $L$ with pilot bandwidth $\bG$, possibly different from $K$ and $\bH$. For the
selection of the pilot bandwidth matrix $\bG$, the same authors showed that the leading term of the
mean squared error $\E\big[ \| \hat{\boldsymbol{\psi}}_{2r} (\bG) - \boldsymbol{\psi}_{2r}
\|^2\big]$ is given by the squared norm of the asymptotic bias vector
\begin{equation}
\label{eq:omega_pi}
\boldsymbol{\omega}_{\PI,2r}(\bG) = n^{-1} |\bG|^{-1/2} (\bG^{-1/2})^{\otimes 2r} \D^{\otimes 2r} L(0)
+ \tfrac{m_2(L)}{2} (\vect \bG \otimes \bI_{d^{2r}}) \boldsymbol{\psi}_{2r+2},
\end{equation}
so that the asymptotically optimal choice of the pilot bandwidth for the estimation of $\bpsi_{2r}$ is $\bG_{\PI,2r} = \mathrm{argmin}_{\bG\in \mathcal{F}} \lVert
\boldsymbol{\omega}_{\PI,2r}(\bG) \lVert^2$, which depends on $\bpsi_{2r+2}$.

Hence, to select the pilot bandwidth $\bG$ from the data we could substitute $\bpsi_{2r+2}$ by
another kernel estimator in (\ref{eq:omega_pi}) and minimize the squared norm of the resulting
vector, but of course then the optimal bandwidth for the kernel estimator of $\psi_{2r+2}$ depends
on $\bpsi_{2r+4}$, and so on. The usual strategy to overcome this problem is to use an $m$-stage
algorithm as described in \cite{CD10}, involving $m$ successive kernel functional estimations with
the initial bandwidth matrix chosen on the basis of a normal scale approximation. The resulting
bandwidth obtained by minimizing $\PI_r(\bH)$ when $\bpsi_{2r+4}$ is estimated using an $m$-stage
algorithm is called an $m$-stage plug-in bandwidth selector for the $r$-th derivative.

\color{black}

\subsection{Smoothed cross validation method}
\label{sec:scv}

The smoothed cross validation (SCV) methodology for univariate density estimation was introduced in
\cite{HMP92}, and a thorough study of this technique was made in \cite{JMP91}. However, it has not
been until recently that its multivariate counterpart has been developed, in \cite{DH05b} and
\cite{CD11}, and its use for univariate density derivative estimation has been explored
\citep[see][]{DR10}.

A possible derivation of the SCV criterion {for the problem of multivariate density derivative
estimation} is based on the approximation of the MISE obtained by replacing the exact integrated
variance in equation (\ref{exactMISE}) by its asymptotic approximation (the first term), while
keeping the exact form for the integrated squared bias, so that $\MISE_r(\bH)\simeq\MISE2_r(\bH)$
with
\begin{align*}
\MISE2_r(\bH) &= n^{-1}|\bH|^{-1/2}\tr\big((\bH^{-1})^{\otimes r}\mathbf R(\D^{\otimes r}K)\big)
+  \tr \mat{R}^*( \bar{\Delta}_\bH, \D^{\otimes r}f),
\end{align*}
where $\Delta_\bH = K_\bH - K_0$ (here $K_0$ denotes the Dirac delta function) and
$\bar{\Delta}_\bH = \Delta_\bH
* \Delta_\bH = \bar{K}_\bH - 2K_\bH + K_0$. The
SCV criterion is obtained by replacing the unknown target $\D^{\otimes r} f$ in the MISE2 formula
with a pilot estimator $\D^{\otimes r} \hat{f}_{\bG}(\bx)=n^{-1}\sum_{i=1}^n\D^{\otimes
r}L_\bG(\bx-\bX_i)$, leading to
\begin{align*}
\mathrm{SCV}_r (\bH)
&=n^{-1}|\bH|^{-1/2}\tr\big((\bH^{-1})^{\otimes r}\mathbf R(\D^{\otimes r}K)\big) \\
&\quad+ (-1)^r n^{-2} (\vect \bI_{d^r})
\sum_{i,j=1}^n \bar{\Delta}_\bH * \D^{\otimes 2r} \bar{L}_\bG (\bX_i - \bX_j),
\end{align*}
where $\bar L=L*L$. When all the $\bX_i$ are distinct and the diagonal terms ($i=j$) are omitted in
the previous sum it can be shown, using the properties of the Dirac delta function \citep[see,
e.g.,][Chapter I.2]{GS64}, that the SCV criterion coincides with the CV criterion for $\bG=0$.

The minimizer of $\SCV_r(\bH)$ is defined to be $\hat{\bH}_{\SCV,r}$. Its value depends on the
pilot selector $\bG$. \citet{CD11} showed that in the case $r=0$ the leading term of the mean
squared error $\E \| \vec (\hat{\bH}_{\SCV,r} - \bH_{\MISE,r}) \|^2$ is given by the squared norm
$\lVert \boldsymbol{\omega}_{\SCV,2r+4}(\bG) \lVert^2$ where $\boldsymbol{\omega}_{\SCV,2r+4}(\bG)$
is the same as the aforementioned $\boldsymbol{\omega}_{\PI,2r+4}(\bG)$ except that $L$ is replaced
by $\bar{L}$. Thus it is straightforward to define, analogously to the plug-in algorithm, the
required optimal $k$-th stage pilot bandwidth of an $m$-stage SCV selector.

\subsection{Convergence results}\label{sec:4}

Let $\hat{\bH}_r = \mathrm{argmin}_{\bH \in \mathcal{F}} \widehat{\MISE}_r(\bH)$ be an arbitrary
data-based bandwidth selector, built up on the basis of an estimated criterion
$\widehat{\MISE}_r(\bH)$. Following \citet{DH05a}, $\hat{\bH}_r$ is said to converge to
$\bH_{\MISE,r}$ at relative rate $n^{-\alpha}$ if
$$
\vec (\hat{\bH} - \bH_{\MISE,r}) = O_P(n^{-\alpha} \mat{J}_{d^2}) \vec \bH_{\MISE,r}
$$
where $O_P$ denotes element-wise order in probability and $\mat{J}_{d^2}$ is the $d^2\times d^2$
matrix of ones. This order in probability statement can be difficult to derive directly. The next
lemma provides a more tractable indirect method of calculating convergence rates.

\begin{lemma}
\label{lem:asymHr} Suppose that assumptions (A1)--(A3) given in the appendix hold. The discrepancy
$\vec(\hat\bH_r-\bH_{\MISE,r})$ is asymptotically equivalent to
$\D_\bH[\widehat{\mathrm{MISE}}_r-\mathrm{MISE}_r](\bH_{\MISE,r})$, where $\D_\bH$ is shorthand for
$\partial/\partial\vec\bH$. Furthermore, the relative rate of convergence of $\hat\bH_r$ is
$n^{-\alpha}$ if
\begin{align*}
\E \lbrace \D_\bH[&\widehat{\mathrm{MISE}}_r-\mathrm{MISE}_r](\bH_{\MISE,r}) \rbrace
\E \lbrace \D_\bH[\widehat{\mathrm{MISE}}_r-\mathrm{MISE}_r](\bH_{\MISE,r}) \rbrace^\top \\
&+ \Var \lbrace \D_\bH[\widehat{\mathrm{MISE}}_r-\mathrm{MISE}_r](\bH_{\MISE,r}) \rbrace
= O(n^{-2\alpha} \mat{J}_{d^2}) \vec \bH_{\MISE,r}  \vect \bH_{\MISE,r}.
\end{align*}
\end{lemma}
The convergence rates of the three bandwidth selectors considered here are given in the following
theorem, whose proof is deferred to the appendix.

\begin{theorem}\label{ReRoC}
Suppose that assumptions (A1)--(A5) given in the appendix hold. The relative rate of convergence to
$\bH_{\MISE,r}$ is $n^{-d/(2d+4r+8)}$ for the cross validation selector $\hat{\bH}_{\CV,r}$, and
$n^{-2/(d+2r+6)}$ for the plug-in selector $\hat{\bH}_{\PI,r}$ and the smoothed cross validation
selector $\hat{\bH}_{\SCV,r}$ when $d\geq2$.
\end{theorem}

\cite{Jon92} computed the relative rate of convergence for the CV and PI selectors for the
estimation of a single partial derivative, using a single-parameter bandwidth matrix (i.e.,
$\bH\in\mathcal I$). The previous theorem shows that the unconstrained CV bandwidth attains the
same rate as its constrained counterpart, yet with added flexibility that should be captured in the
constant coefficient of the asymptotic expression, although the computation of an explicit form for
this coefficient does not seem possible in general.

The convergence rate of the PI selector is $n^{-(2+\min\{2,d/2\})/(d+2r+6)}$ within the
single-parameter bandwidth class $\mathcal I$, yielding a slightly faster convergence to the
optimal constrained bandwidth. As explained in \cite{CD10,CD11} for the density case, this is due
to the fact that the very special cancellation in the bias term which is achievable when using a
single-parameter bandwidth is not possible in general for the unconstrained estimator.
Nevertheless, the aforementioned papers showed that this slight loss in convergence rate terms is
negligible in practice as compared with the fact that the targeted constrained optimal bandwidth is
usually much less efficient than the unconstrained one (see also Section \ref{sec:5} below).

Theorem \ref{ReRoC} also shows that the similarities noted in \cite{CD11} about the asymptotic
properties of the PI and SCV methods for the density estimation problem persist for $r>0$, since
both selectors exhibit the same relative rate of convergence.

\citet[][p. 406]{JMS96} exemplified how slow is the rate $n^{-1/10}$ of the CV selector for $d=1$,
$r=0$ by noting that $n$ has to be as large as $10^{10}=10,000,000,000$ so that $n^{-1/10}=0.1$. In
the same spirit, to compare the rates obtained in Theorem \ref{ReRoC}, Table \ref{rates-ratio}
shows the values of $n^{-d/(2d+4r+8)}$ (CV) and $n^{-2/(d+2r+6)}$ (PI and SCV) divided by
$1000^{-1/6}$, that is the rate for the CV selector $n=1000, d=2, r=0$ which is used as a base
case, for all the different combinations of $n=10^3, 10^4, 10^5$, $d=2,3,4,5$ and $r=0,1,2$. Ratios
which are lower than 1 indicate the rate is faster than the base case, and ratios greater than 1 a
slower rate. For $n=10^3$, these ratios in Table \ref{rates-ratio} tend to be greater than 1,
indicating that using this sample size will lead to a deteriorating convergence rate. On the other
hand for the larger sample sizes, $n=10^4, 10^5$, these ratios tend to be less than 1. This implies
that convergence rates better than the CV selector for bivariate density estimation can be
attained, even with higher dimensions and higher order derivatives, provided that sufficiently
large (although still realistic) sample sizes are used. Of course this comparison only takes into
account the asymptotic order of the convergence rates by ignoring the associated coefficients since
explicit formulas for the latter are not available for $d\geq 2$. The finite sample behaviour of
the bivariate case for moderate sample sizes is examined more closely in the next section.

\begin{table}[t]
\centering
\begin{tabular}{llc@{}cc@{}cc@{}cc@{}c}
\hline
        & & \multicolumn{2}{c}{$d=2$} & \multicolumn{2}{c}{$d=3$} & \multicolumn{2}{c}{$d=4$} & \multicolumn{2}{c}{$d=5$}  \\
\hline
        & & CV & PI/SCV & CV & PI/SCV & CV & PI/SCV & CV & PI/SCV \\
$r=0$
& $n=10^3$ & 1.000 & 0.562 & 0.720 & 0.681 & 0.562 & 0.794 & 0.464 & 0.901 \\
& $n=10^4$ & 0.681 & 0.316 & 0.439 & 0.408 & 0.316 & 0.501 & 0.245 & 0.593 \\
& $n=10^5$ & 0.464 & 0.178 & 0.268 & 0.245 & 0.178 & 0.316 & 0.129 & 0.390 \\
        \hline
$r=1$
& $n=10^3$ & 1.334 & 0.794 & 1.000 & 0.901 & 0.794 & 1.000 & 0.658 & 1.093 \\
& $n=10^4$ & 1.000 & 0.501 & 0.681 & 0.593 & 0.501 & 0.681 & 0.390 & 0.767 \\
& $n=10^5$ & 0.750 & 0.316 & 0.464 & 0.390 & 0.316 & 0.464 & 0.231 & 0.538 \\
        \hline
$r=2$
& $n=10^3$ & 1.585 & 1.000 & 1.233 & 1.093 & 1.000 & 1.179 & 0.838 & 1.259 \\
& $n=10^4$ & 1.259 & 0.681 & 0.901 & 0.767 & 0.681 & 0.848 & 0.538 & 0.926 \\
& $n=10^5$ & 1.000 & 0.464 & 0.658 & 0.538 & 0.464 & 0.611 & 0.346 & 0.681\\
       \hline
 \end{tabular}
\caption{Comparison of the relative rate of convergence for the CV, PI and SCV selectors. For each
combination of $r$, $n$ and $d$ in the table, the left entry in the corresponding cell shows
$n^{-d/(2d+4r+8)}$ (CV selector) and the right entry $n^{-2/(d+2r+6)}$ (PI and SCV selectors)
divided by $1000^{-1/6}$ (i.e. the rate for the CV selector with $n=1000, d=2, r=0$).}
\label{rates-ratio}
\end{table}
\color{black}

\section{Numerical study}\label{sec:5}

\subsection{Data-based algorithms}
For most practical implementations the normal kernels are used, i.e. $K=L=\phi$. For $d \times d$
symmetric matrices $\mat A, \mat B$, and for $r, s \geq 0$, let
$$\eta_{2r,2s}(\bx; \mat A, \mat B, \bSigma) = [(\vec^T \mat A)^{\otimes r} \otimes
(\vec^T \mat B)^{\otimes s}] \D^{\otimes 2r+2s} \phi_{\bSigma} (\bx)$$
and write, for short, $\eta_{2r}(\bx; \bSigma) = \eta_{2r,0}(\bx; \bI_d, \bI_d, \bSigma)$ and $\nu_r(\bSigma)=(-1)^r\eta_{2r}(0;\bSigma)/\phi_\bSigma(0)$.

Then the cross validation criterion can be rewritten as
\begin{align*}
\CV_r(\bH)&=(-1)^r \bigg\{n^{-2}\sum_{i,j=1}^n \eta_{2r} (\bX_i-\bX_j;  2\bH)
-2[n(n-1)]^{-1}\sum_{i\neq j} \eta_{2r} (\bX_i-\bX_j; \bH) \bigg\}.
\end{align*}
Besides, the data-based $m$-stage selection algorithm for plug-in selectors is given by:
\begin{enumerate}
\item Initialize the $m$-th stage pilot selector to be the normal reference selector
$$\displaystyle \hat{\bG}_{\PI,2r+2m+2} = \bigg( \frac{2}{2r+2m+d+2} \bigg)^{2/(2r+2m+d+4)} 2 \mat S n^{-2/(2r+2m+d+4)},$$
from \citet{CDW11}, where
$\mat{S}$ is the sample variance of $\bX_1,\dots \bX_n$.
\item For $k=m-1, m-2, \dots, 1$, the optimal $k$-the stage pilot selector $\hat{\bG}_{\PI,
    2r+2k+2}$ is the minimizer of
\begin{align*}
\lVert \hat{\boldsymbol{\omega}}_{\PI,2r+2k+2} (\bG)\lVert^2 &= n^{-2} |\bG|^{-1} (2\pi)^{-d} \mathrm{OF}(2r+2k+2) \nu_{r+k+1}(\bG^{-2}) \\
 &+ (-1)^{r+k+1} (2\pi)^{-d/2} \mathrm{OF}(2r+2k+2)
|\bG|^{-1/2} n^{-3}\\&\quad\times\sum_{i,j=1}^n \eta_{2,2r+2k+2}(\bX_i - \bX_j; \bG, \bG^{-1}, \hat{\bG}_{\PI,2r+2k+4}) \\
 &+  \tfrac{1}{4} n^{-4} \Big[ \sum_{i,j=1}^n \eta_{2,2r+2k+2} (\bX_i - \bX_j;\bG, \bI_d, \hat{\bG}_{\PI,2r+2k+4}) \Big]^2,
\end{align*}
where ${\rm OF}(2p)=(2p-1)(2p-3)\cdots5\cdot3\cdot1$ for $p\in\mathbb N$. The numerical
minimization over the class of positive-definite matrices is carried out as described in detail
in \citet[][Section 5.1]{DH05b}.

\item The plug-in selector $\hat{\bH}_{\PI,r}$ is the minimizer of
\begin{multline*}
\mathrm{PI}_r(\bH) = n^{-1}|\bH|^{-1/2} 2^{-(d+r)} \pi^{-d/2} \nu_r(\bH^{-1})\\
+ (-1)^r (2n)^{-2}\sum_{i,j=1}^n \eta_{2, 2r}(\bX_i - \bX_j; \bH, \bI_d, \hat{\bG}_{\PI, 2r+4}).
\end{multline*}
\end{enumerate}
The derivations of $\lVert \hat{\boldsymbol{\omega}}_{\PI,2r+2k+2} (\bG)\lVert^2$ and
$\mathrm{PI}_r(\bH)$ in the $\eta$ functional form can be found in \citet{CD12}. There it is also
shown that, although it appears that these are less concise than the previous expressions, they
facilitate efficient computation, both in terms of memory and execution time.

We observe that $\lVert \boldsymbol{\omega}_{\SCV, 2r+2k+2} \lVert^2$ is the same as $\lVert
\boldsymbol{\omega}_{\PI, 2r+2k+2} \lVert^2$ except the three terms are multiplied by $2^{-d},
2^{-d/2+1}$ and 4 respectively; since $\bar{\phi}(0) = \phi_{2\bI}(0) = 2^{-d/2} \phi(0)$ and
$m_2(\bar{\phi}) = 2m_2(\phi)$. Furthermore,
\begin{align*}
\mathrm{SCV}_r(\bH) &= n^{-1}|\bH|^{-1/2} 2^{-(d+r)} \pi^{-d/2} \nu_r(\bH^{-1})
+ (-1)^r n^{-2} \sum_{i,j=1}^n \big[\eta_{2r}(\bX_i - \bX_j; 2\bH+2\bG) \\
&\quad -2\eta_{2r}(\bX_i - \bX_j; \bH+2\bG) + \eta_{2r}(\bX_i - \bX_j; 2\bG)\big].
\end{align*}
So a data-based $m$-stage SCV selector is obtained from straightforward modifications of the PI
selector algorithm above.

\subsection{Simulation study}

The bandwidth selectors included in our simulation study were
\begin{itemize}
\item OR: oracle, i.e. the minimizer of the MISE for the target density
\item NR: normal reference from \citet[Theorem~6]{CDW11}, which is equal to
$[4/(d + 2r + 2)]^{2/(d+2r+4)} \mat{S} n^{-2/(d+2r+4)}$
\item CV: cross validation from Section~\ref{sec:cv}

\item PI: plug-in with 2-stage unconstrained pilots from Section~\ref{sec:plugin}
\item SCV: smoothed cross validation with 2-stage unconstrained pilots from
    Section~\ref{sec:scv}
\end{itemize}
We have developed efficient implementations of all these selectors and incorporated them into the
existing R library \texttt{ks} \citep{Du07}. The target bivariate normal mixture densities that we
considered are displayed in Figure~\ref{fig:nm}. Their explicit definitions can be found in
\cite{Ch09}.

\begin{figure}[!htp]
\includegraphics[width=\textwidth]{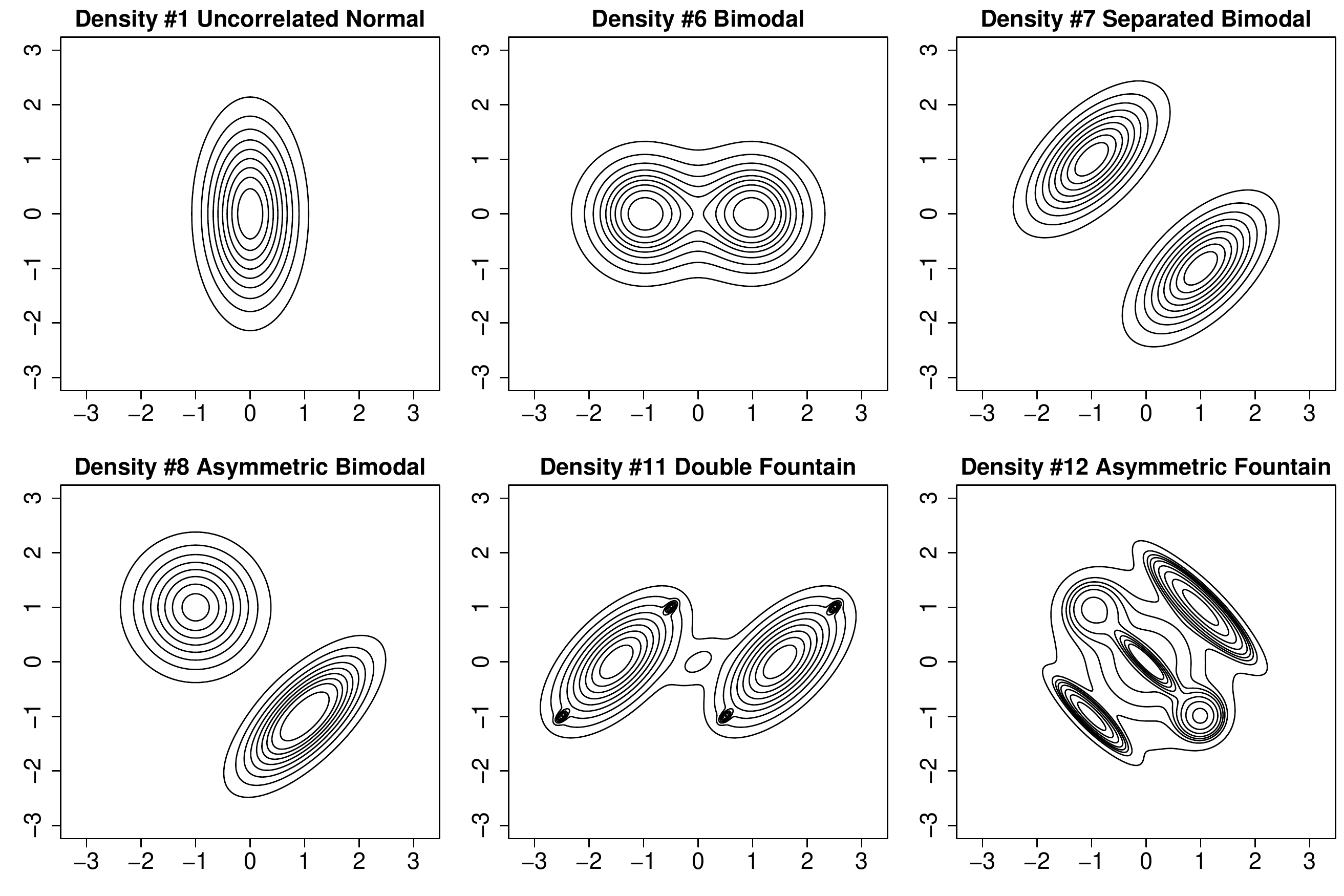}
\caption{Target bivariate normal mixture densities}
\label{fig:nm}
\end{figure}

For each selector and target density and for $r=0,1,2$, we generated 100 samples of size $n=1000$.
The integrated squared error (ISE) between the resulting density estimate and the target density
was computed as our measure of performance. The box plots of the log ISE are shown in
Figure~\ref{fig:ise-samp1000}. We also conducted the study for sample sizes $n=400$ and $n=4000$
but the conclusions extracted were much the same as for $n=1000$ so we decided not to include these
results here to avoid redundancies.

By construction, the oracle selector (OR) is the best possible selector in terms of MISE given that
the true target normal mixture density was used for its computation. As expected, it also has the
uniformly lowest ISE. The normal reference selector (NR) was the only data-based selector
previously available in the literature, and the results show that it is suitable only for
density~\#1 since its ISEs for the other densities are uniformly higher than those of the other
selectors.  In line with other published simulation studies {\citep{CCGM94,JMS96}}, the CV selector
displays larger variability in the ISEs than the PI and SCV selectors, though the former presents
lower mean ISEs in some cases, \color{black} e.g. density~\#12, $r=0,1$.
We note also that the CV
variability tends to increase with increasing $r$, whereas this is not observed for the two other
hi-tech selectors. An anonymous referee drew our attention to the low variability of the introduced bandwidth selectors for this density \#12, as compared with that of the oracle. This density has very complicated features, like modal regions of different shape and size, so this is the scenario where
usually oversmoothing occurs, and we checked that this is indeed the case: the oracle tries hard to discover the true structure (hence its high variability), whereas all the data-driven bandwidths tend to consistently prefer a more conservative estimate, slightly oversmoothed.  Given that the construction {and theoretical properties} of the PI and SCV
selectors are similar, it is not surprising that their ISE performance is correspondingly similar
for all the cases examined here. Either of these selectors would thus be our recommendation over
the CV and NR selectors.


\begin{figure}[!htp]
\includegraphics[width=\textwidth]{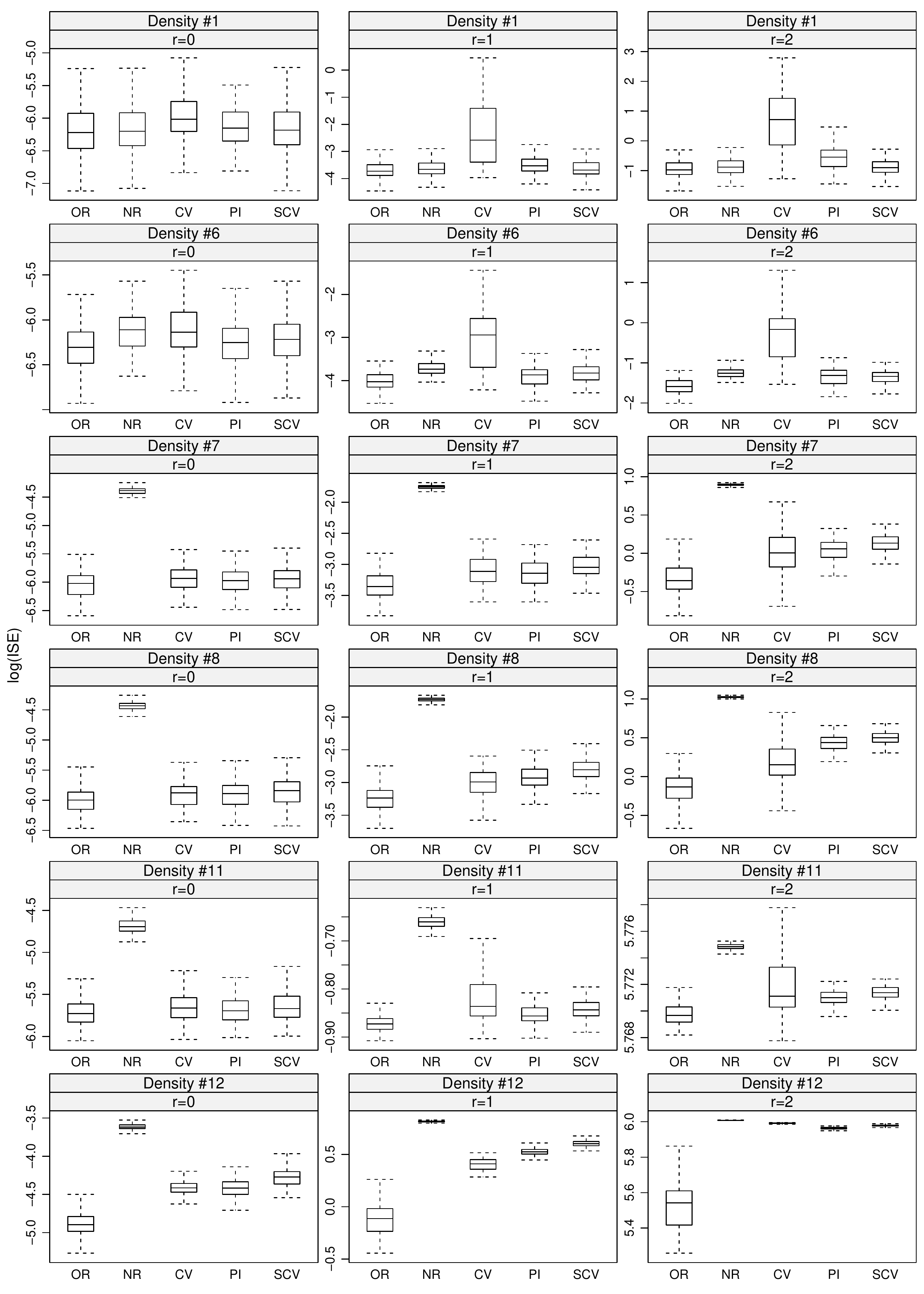}
\caption{Box plots of the logarithm of the ISEs for bandwidth selectors for $n=1000, r=0,1,2$ for the six
bivariate target densities.}
\label{fig:ise-samp1000}
\end{figure}


\section{Applications to mean shift clustering}\label{sec:ms}

The so-called mean shift algorithm \citep{FH75} is an iterative procedure which, at every step,
shifts the point obtained in the previous iteration in the direction of the density gradient, producing
a convergent sequence that transports any initial value to a local maximum of the density
along the steepest ascent path.

Specifically, the mean shift clustering algorithm can be described as follows: an initial point $\bY_0$ is transformed recursively to obtain a sequence defined by
\begin{equation}\label{eq:meanshift}
\bY_{j+1}=\bY_j+\mat A\widehat{\D f}(\bY_j)\big/\hat f(\bY_j),
\end{equation}
where $\hat f$ is an arbitrary density estimator, $\widehat{\D f}$ is an estimator of the density gradient, and $\mat A$ is a fixed $d\times d$ positive definite matrix, properly chosen to guarantee convergence of the sequence $(\bY_0,\bY_1,\ldots)$. This is easily recognized as a variant of the well-known gradient ascent algorithm employed to find the local maxima of a given function, but using the normalized density gradient (i.e., the density gradient divided by the density itself) instead of just the gradient in its definition. The advantages of using such a normalization are illustrated in \cite{FH75}, \cite{Ch95} and \cite{CM02}; one of them is to accelerate the convergence of the resulting sequence for initial values of low density.

When kernel estimators are used in (\ref{eq:meanshift}) the above procedure attains a particularly simple form. Assuming that the kernel $K$ is a spherically symmetric function it follows that $K(\bx)=\tfrac12k(\|\bx\|^2)$, where the function $k\colon\mathbb R_+\to\mathbb R$ is known as the profile of $K$. Under the usual conditions that $K$ is smooth and unimodal, its profile is decreasing so that $g(x)=-k'(x)\geq0$. Therefore, noting that $\D K(\bx)=-\bx g(\|\bx\|^2)$, the kernel density gradient estimator can be written as
\begin{align}\label{eq:DfH}
\D\hat f_\bH(\bx)&=n^{-1}|\bH|^{-1/2}\bH^{-1}\sum_{i=1}^n(\bX_i-\bx)g\big((\bx-\bX_i)^\top\bH^{-1}(\bx-\bX_i)\big)=\bH^{-1}\tilde f_\bH(\bx)\mat m_\bH(\bx),
\end{align}
where $\tilde f_\bH(\bx)=n^{-1}|\bH|^{-1/2}\sum_{i=1}^ng\big((\bx-\bX_i)^\top\bH^{-1}(\bx-\bX_i)\big)$ can be understood as an unnormalized kernel estimator of $f$ and the term
$$\mat m_\bH(\bx)=\frac{\sum_{i=1}^n\bX_ig\big((\bx-\bX_i)^\top\bH^{-1}(\bx-\bX_i)\big)}{\sum_{i=1}^ng\big((\bx-\bX_i)^\top\bH^{-1}(\bx-\bX_i)\big)}-\bx$$
is known as the {\em mean shift}. Thus, equation (\ref{eq:DfH}) can be re-arranged to note that $\bH^{-1}\mat m_\bH(\bx)$ provides a reasonable estimator of the normalized density gradient, and by taking $\mat A=\bH$ in equation (\ref{eq:meanshift}) it leads to the recursively defined sequence
\begin{equation}\label{eq:ms2}
\bY_{j+1}=\bY_j+\mat m_\bH(\bY_j)=\frac{\sum_{i=1}^n\bX_ig\big((\bY_j-\bX_i)^\top\bH^{-1}(\bY_j-\bX_i)\big)}{\sum_{i=1}^ng\big((\bY_j-\bX_i)^\top\bH^{-1}(\bY_j-\bX_i)\big)}.
\end{equation}
When $k$ is a convex and monotonically decreasing profile, and $\bH=h^2\bI_d$, \citet[Theorem 1]{CM02} showed that the sequence $(\bY_0,\bY_1,\ldots)$ defined in this simple way converges to a local maximum of $\hat f_\bH$, and their proof can be easily adapted to cover the case of an unconstrained $\bH$ as well. The recursive formulation (\ref{eq:ms2}) was also motivated as an EM-type algorithm for mode finding in \cite{LRL07}, who proved its convergence under more general conditions.

Since the direction along which the data points are shifted, as well as the limit points of the
sequences of successive locations (i.e., the solutions of $\D\hat f_\bH(\bx)=0$), are directly
related to the density gradient, our proposal is to take $\bH$ in the mean shift algorithm as a
bandwidth matrix selector for multivariate kernel density gradient estimation, using any of the
methods introduced in Section \ref{sec:3}. This choice is also supported by the results in
\cite{GH95} and \cite{V96}, where it was shown that the optimal bandwidth choice for estimating the
mode of a density is closely related to the problem of density derivative estimation. Thus, the
bandwidth choice is made with the goal of optimal identification of the density features in mind.
This is in contrast with other proposals, as for example \cite{Co03}, where a different criterion
is taken into account to obtain an automatic variable-bandwidth selection algorithm.

When only a few iterations of the mean shift algorithm are performed, it
is probably the case that convergence has not been reached yet. However,
the procedure is still useful for other tasks. These include data
filtering \citep{FH75}, which seeks to reduce the effect of noise in the
determination of the geometric properties of a data set or in finding
local principal curves, and also data sharpening \citep{CH99,HM02}, which
can be used to reduce the bias in kernel curve estimation and to adapt
kernel estimators to pre-specified curve constraints.

The main statistical application of the mean shift procedure is for cluster analysis.
For several additional applications in engineering, see \citet{Ch95,CM02,CRM03}. When the mean
shift algorithm is applied with any of the data points as starting value it induces a partition of
the data in a natural way, by assigning the same cluster to all the data points that converge to
the same local maximum. This is called {\em modal clustering} in \cite{LRL07}. Notice that this
methodology does not require the number of clusters to be specified in advance, and that it allows
clusters of arbitrary shape to be discovered. Moreover, since the mean shift algorithm can be
applied with any starting point, it does not produce only a partition of the data, but a partition
of the whole space.

To illustrate the use of mean shift clustering Figure \ref{fig:paths} shows the result of applying
the mean shift algorithm to a sample of size $n=210$ from a trimodal normal mixture (the one
labeled Trimodal III in \cite{WJ93}). The black bold stars show the location of the three modes
found and the paths in grey starting from every data point depict their
ascent towards their associated density mode.
\begin{figure}[!htp]\centering
\includegraphics[scale=0.65]{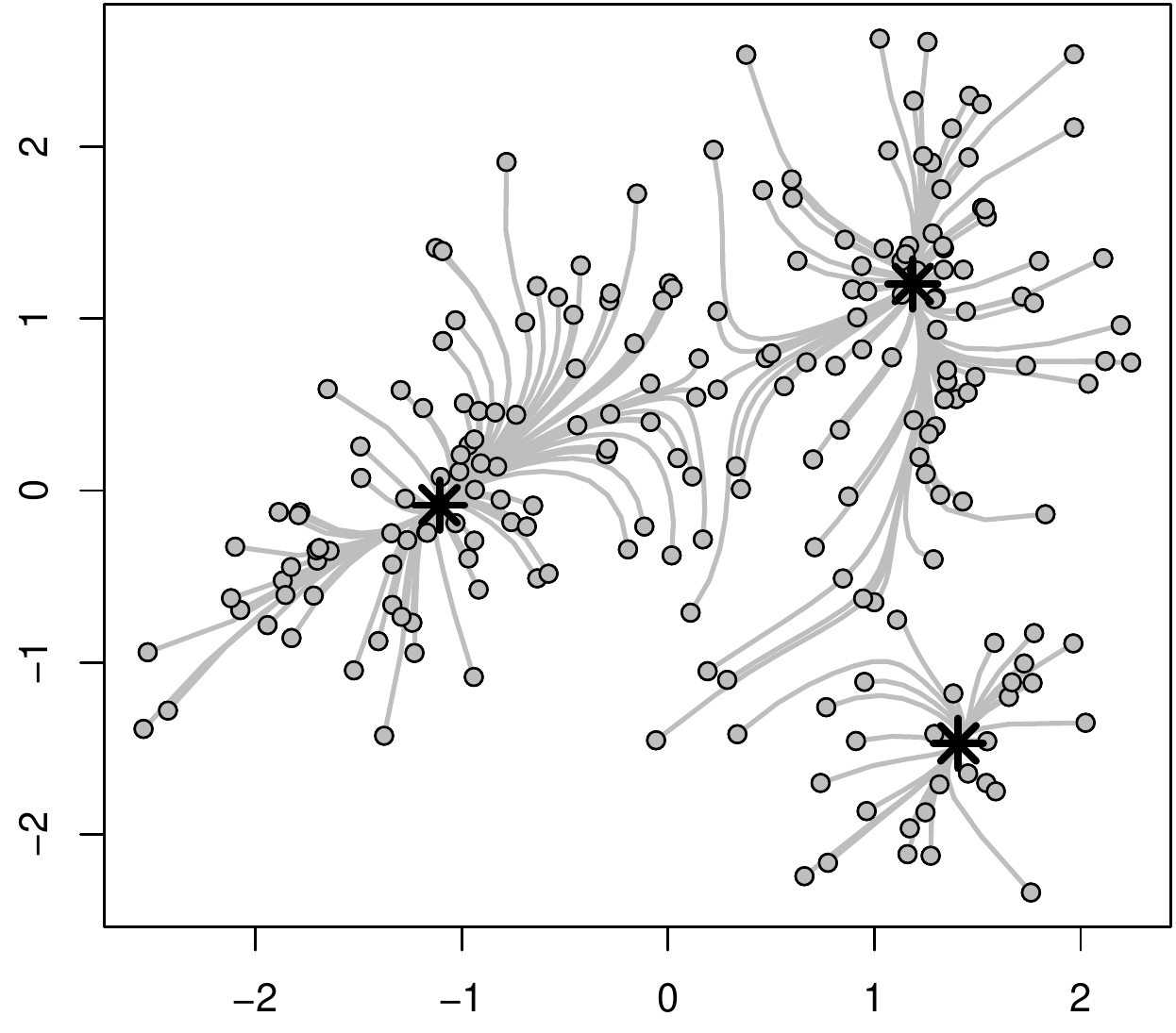}
\caption{Paths followed by the sample points as a result of the application of the mean shift algorithm. Sample of size $n=210$ from a trimodal normal mixture density.}
\label{fig:paths}
\end{figure}

\subsection{Simulation results}

As pointed out above, any of the bandwidth selection methods for kernel density gradient estimation
introduced in Section \ref{sec:3} leads automatically to a new nonparametric clustering procedure
via the mean shift algorithm. To explore the finite sample properties of these new proposals, their
performance is compared here to other related existing methods.

Given the enormous amount of literature on clustering techniques, it would be impossible to include
all the different clustering procedures in this comparison, so a brief selection of techniques
similar to the one introduced here have been considered:
\begin{itemize}
\item CLUES algorithm (CLUstEring based on local Shrinking), proposed in \cite{WQZ07}, is an
    iterative algorithm closely related to the mean shift algorithm, but in which the shift is
    performed at each iteration by computing the coordinate-wise median of the $K$-nearest
    neighbors of the previous iteration point.

\item PDFC algorithm (PDF Clustering), proposed in \cite{AT07}, is also based on a kernel
    density estimate. Its high density regions are computed and the connected components of
    this regions are identified as sample clusters. The bandwidth used in the kernel estimate
    is just a diagonal normal scale rule-of-thumb for the density (not the density gradient),
    multiplied by a subjectively chosen shrinkage factor $3/4$ to correct for oversmoothing. {
    We are aware of the existence of other clustering methods based on high density regions, as
    for instance \cite{CFF01} or \cite{RW10}, but decided to include in this admittedly limited
    study only the PDFC algorithm due to its simplicity.}

\item MCLUST algorithm (Mixture model CLUSTering), as surveyed in \cite{FR02}, is included in
    the comparison since it can be recognized as the parametric golden standard.
\end{itemize}
These three methodologies are compared with mean shift clustering using unconstrained bandwidth
matrices for density gradient estimation obtained with: 1) the normal scale rule derived in
\cite{CDW11} (labeled NR), 2) the cross-validation bandwidth (labeled CV), 3) the plug-in bandwidth
(labeled PI), and 4) the smoothed cross-validation bandwidth (labeled SCV).

The comparison is made along five test clustering problems, generated by five bivariate mixture
densities that have been chosen to investigate the performance of the methods in a typical
parametric setup (two normal mixture densities) and in situations with non-ellipsoidal cluster
shapes, having also different scales. Specifically, the five mixture densities in the study are:
\begin{enumerate}
\item Trimodal III density from \cite{WJ93}.
\item Quadrimodal density from \cite{WJ93}.
\item 4-crescent model. This model is intended to mimic the distribution explored in Figure 7
    of \cite{Co03}. Since an explicit expression of the density function is not given there,
    our model has been generated as a suitable modification of Experiment 4 in \citet[p.
    546]{F90}. Namely, a bivariate random vector $\bX$ is defined to have a crescent
    distribution with center $\mat O\in\mathbb R^2$, radius $r>0$ and convexity indicator
    $\kappa\in\{0,1\}$, denoted $C(\mat O,r,\kappa)$ if $\bX=\mat O+(r\cos\Theta,(-1)^\kappa
    r\sin\Theta)^\top+\mat U$, where $\Theta$ is normally distributed with mean $\pi/2$ and
    variance $(\pi/6)^2$ and $\mat U$ is a bivariate centred normal vector with variance matrix
    $(r/20)^2\bI_2$. Then, the 4-crescent model is the equally weighted 4-component mixture
    density with components $C((-1,1)^\top,1,1)$, $C((0,0.5)^\top,1,0)$, $C((0,0)^\top,0.5,1)$
    and $C((0.5,-0.5)^\top,0.5,0).$
\item Broken ring model. This model aims to reproduce the sampling scheme shown in Figure 3 in
    \cite{WQZ07}. Precisely, a bivariate random vector $\bX$ is defined to have a standard
    half-crescent distribution with mean angle $\theta$, denoted $HC(\theta)$ if
    $\bX=(\cos\Theta,\sin\Theta)^\top+\mat U$, where $\Theta$ is normally distributed with mean
    $\theta$ and variance $(\pi/12)^2$ and $\mat U$ is a bivariate centred normal vector with
    variance matrix $(1/20)^2\bI_2$. Then, the broken ring model is the 5-component mixture
    density having a centred normal component with variance $(1/5)^2\bI_2$ and weight $1/4$,
    and four standard half-crescent components with equal weights $3/16$ and mean angles
    $\pi/4$, $3\pi/4$, $5\pi/4$ and $7\pi/4$, respectively.
\item Eye model. This model is a variation of the former. It is also a 5-component mixture density with a centred normal component with variance $(1/5)^2\bI_2$ as before, but with a weight $1/20$. The other 4 components are centred crescent distributions (i.e., $\mat O=(0,0)^\top$), two of them with radius 1 and the two possible convexity indicators, respectively, having weight $1/8$ each; and the other two with radius $1.5$ and also the two possible convexity indicators, but with weight $7/20$ each, and rotated 90 degrees.
\end{enumerate}
A clearer picture of all these models is provided by Figure \ref{fig:clustersamp}, which shows samples of size $n=800$ for each of them.

\begin{figure}[t]\centering
\includegraphics[width=0.32\textwidth]{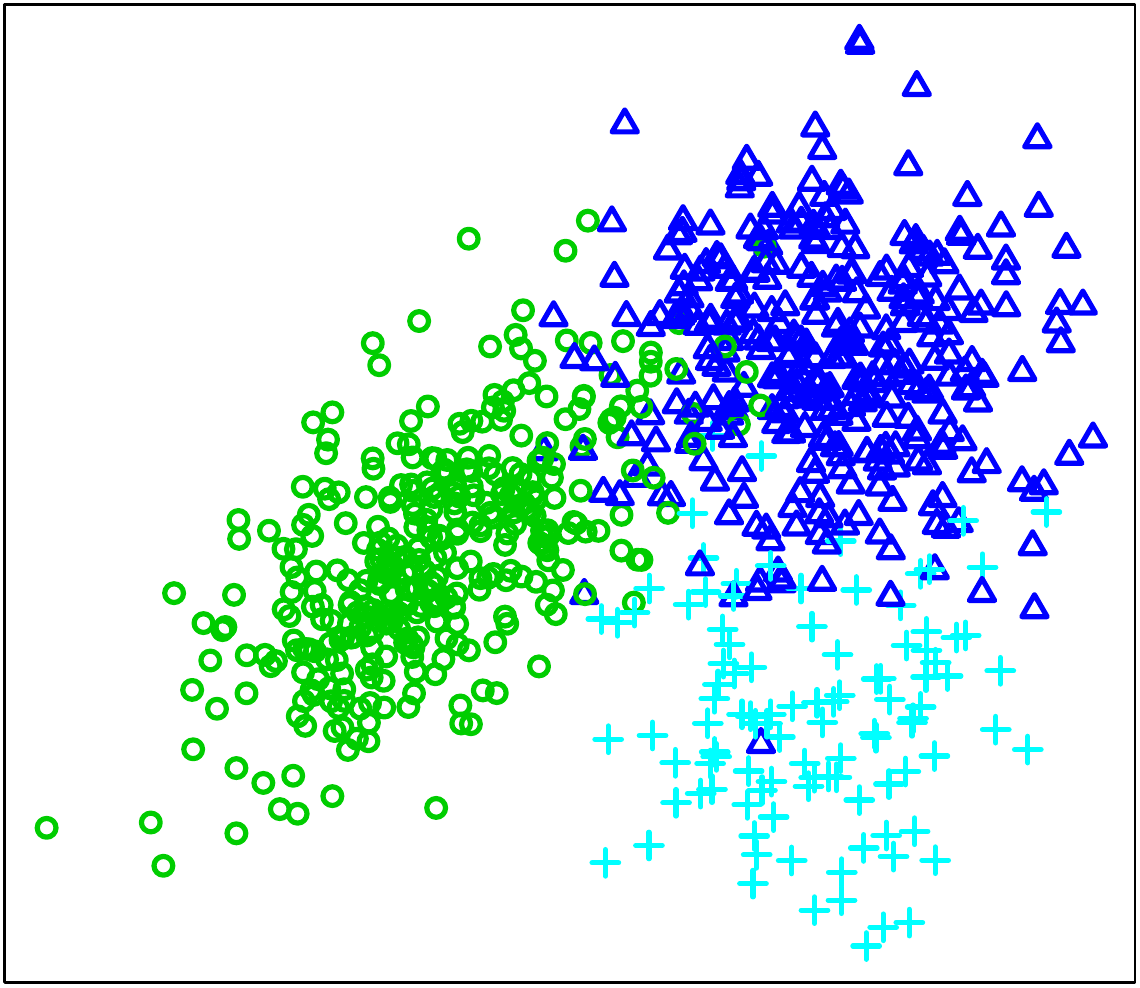}\quad\includegraphics[width=0.32\textwidth]{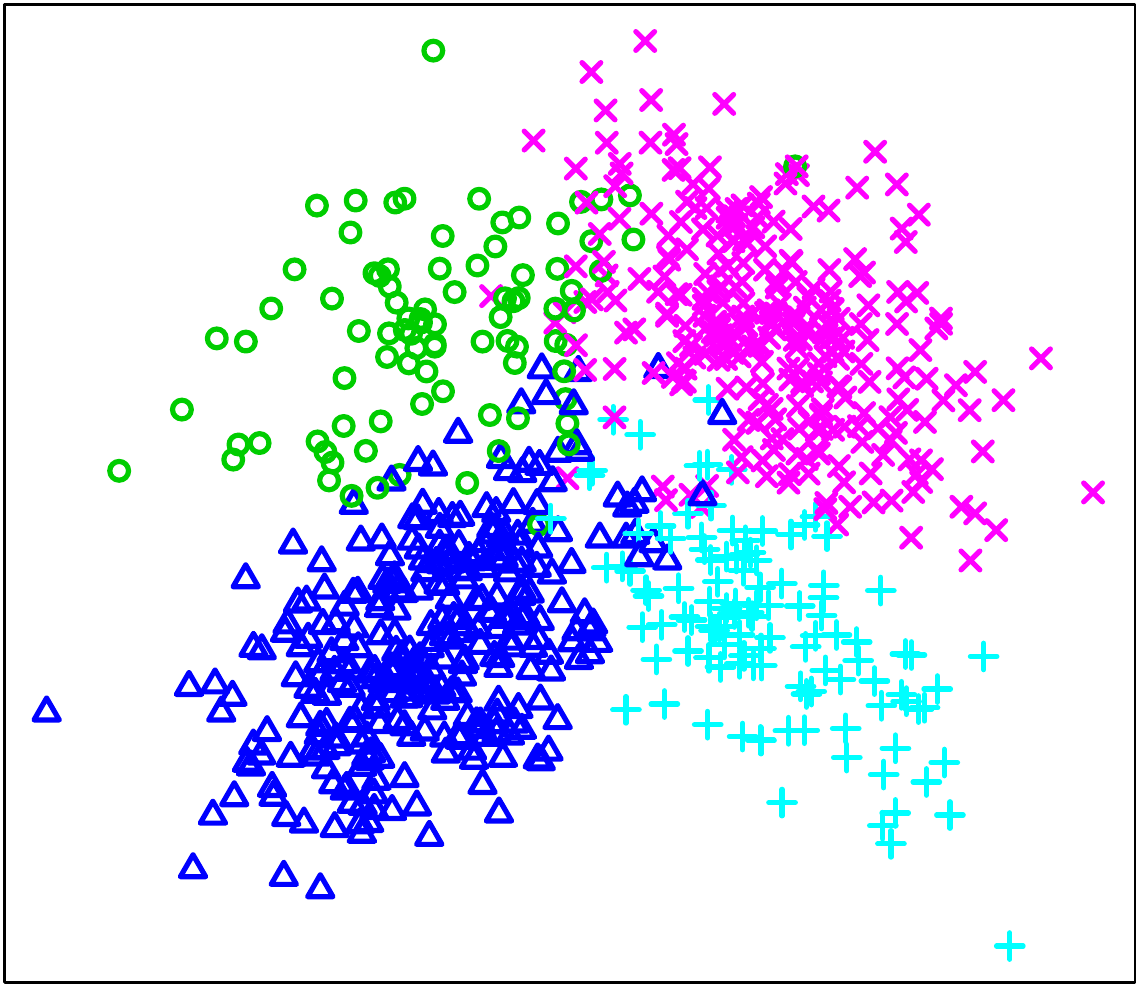}
\includegraphics[width=0.32\textwidth]{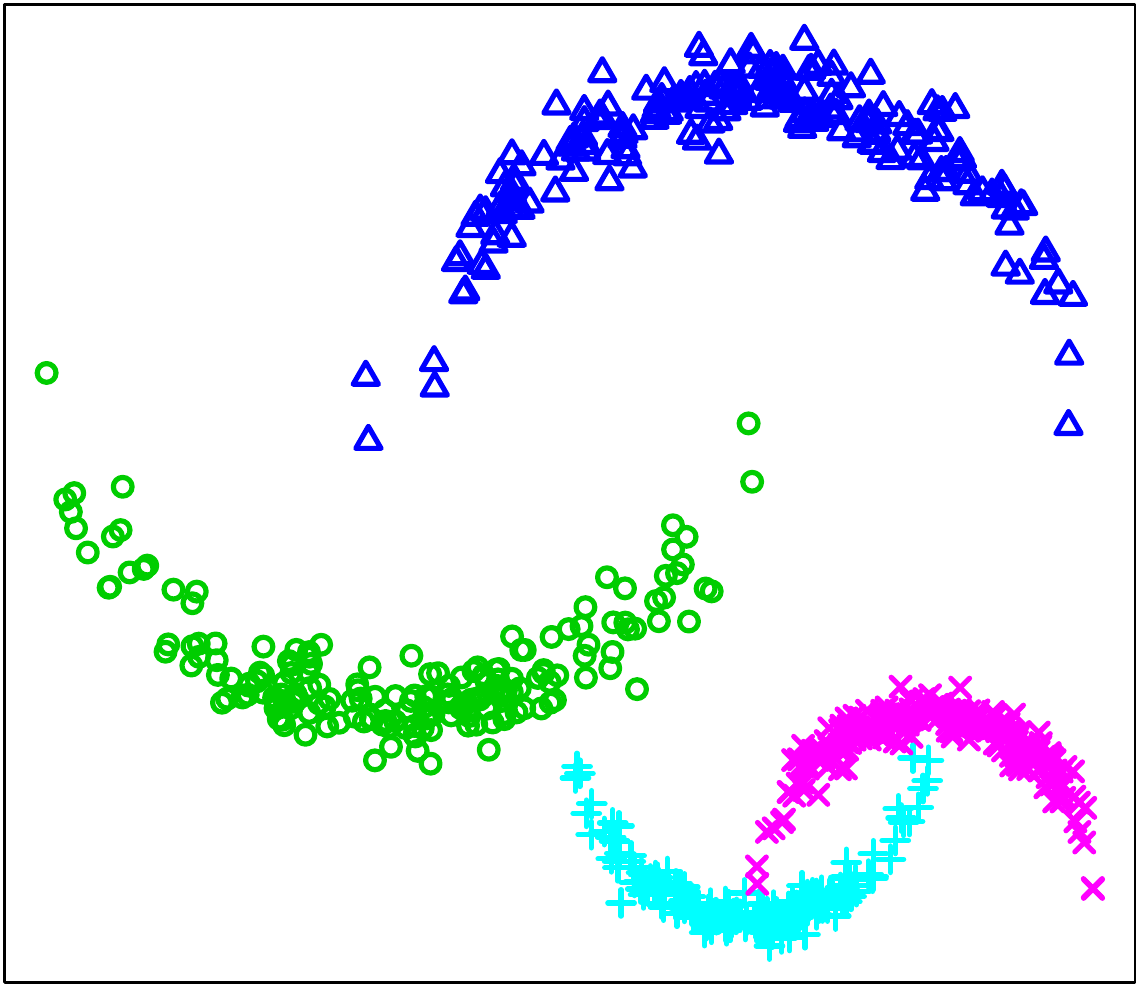}\\[1em]
\includegraphics[width=0.32\textwidth]{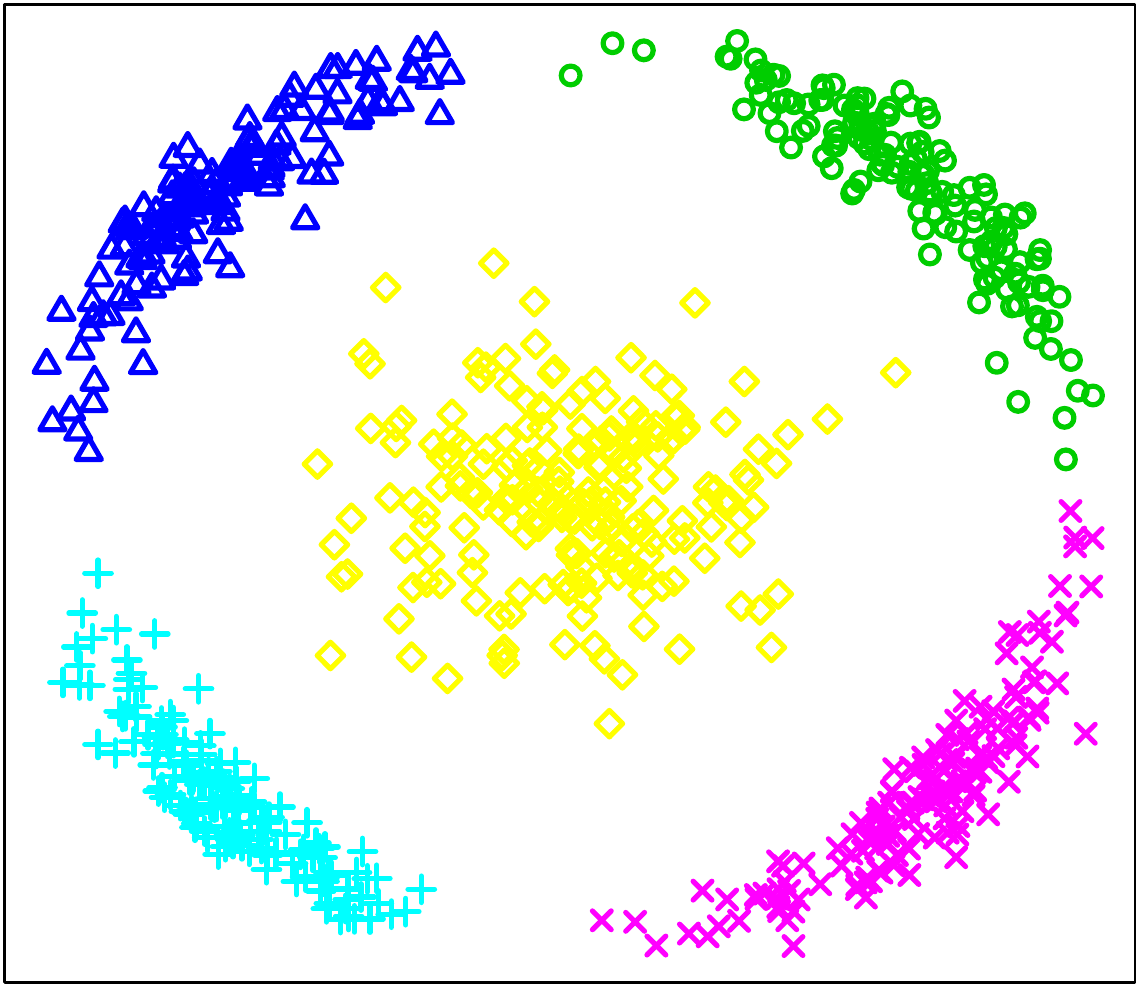}\quad\includegraphics[width=0.32\textwidth]{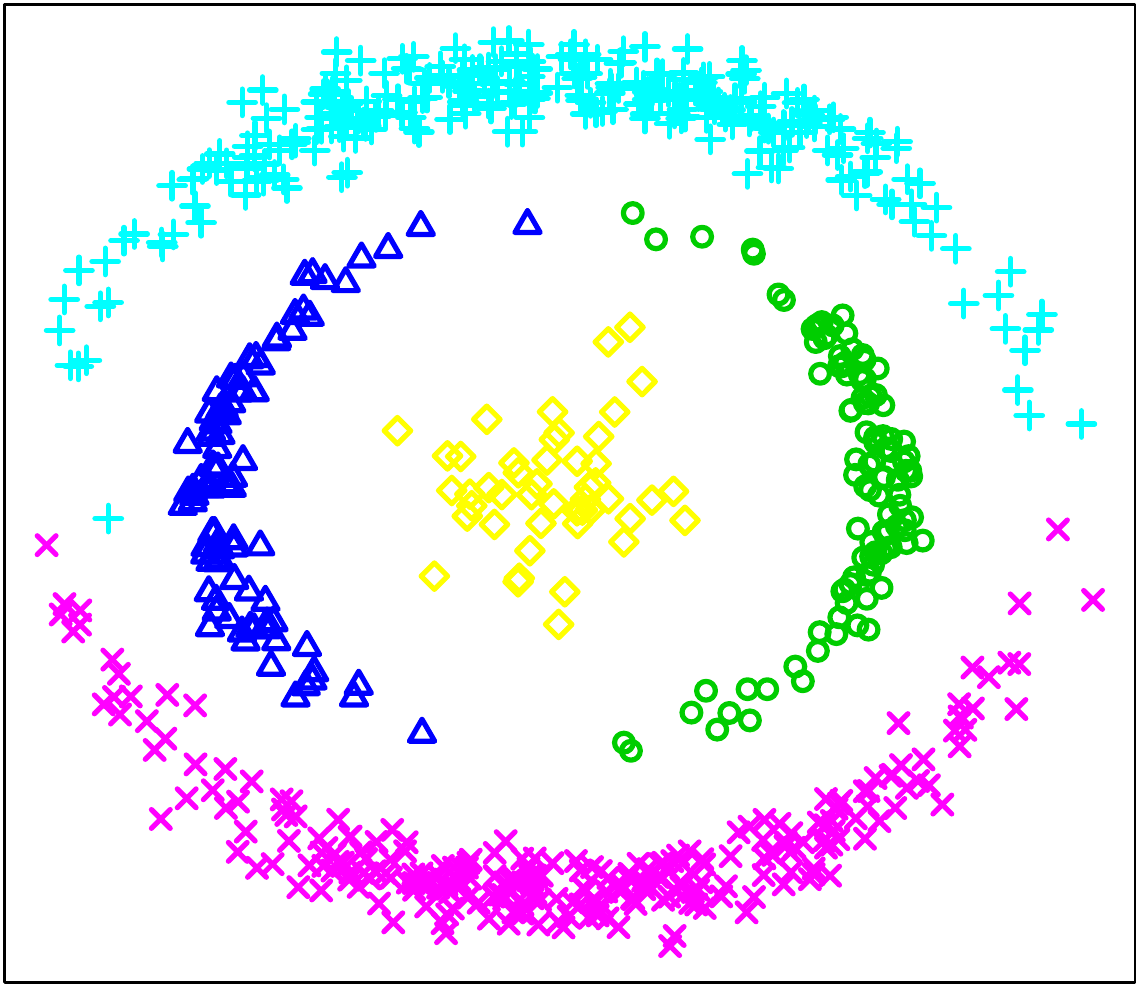}
\caption{Samples of size $n=800$ from each of the models considered in the clustering simulation study.
 Different cluster membership is indicated with different plotting characters and colours.}
\label{fig:clustersamp}
\end{figure}

In common with \cite{AT07}, \cite{WQZ07} and many others, the performance of each clustering method
is measured through the adjusted Rand index (ARI), which was introduced by \cite{HA85} as a
corrected-for-chance version of the proportion of agreements between two partitions of a given data
set. This index is the overall preferred accuracy measure in the simulation study of \cite{MC86}.
An ARI value of 1 indicates that all estimated memberships are the same as the true memberships,
whereas a value close to 0 indicates that the estimated cluster assignation does not differ much
from random assignment. For the comparison, 100 samples of size 500 were drawn from each of the
five test models, the data were clustered according to the seven methods in the study (four mean
shift procedures plus CLUES, PDFC and MCLUST) and the ARI was computed to measure the performance
of each method for each of these data sets. Table \ref{tab:ari} presents the average ARI values
obtained.

\begin{table}[!ht]
\centering
\begin{tabular}{lccccccc}
\hline
         & NR & CV & PI & SCV & CLUES & PDFC & MCLUST\\
\hline

Trimodal III
&   0.700 & 0.694 & 0.752 & 0.546 & 0.583 & 0.754 & 0.715\\
       \hline
Quadrimodal
&   0.518 & 0.617 & 0.630 & 0.505 & 0.519 & 0.641 & 0.790\\
       \hline
4-crescent
&  0.569 & 0.920 & 0.913 & 0.932 & 0.805 & 0.834 & 0.481 \\
       \hline
Broken ring
&   0.983 & 0.918 & 0.983 & 0.986 & 0.984 & 0.975 & 0.811\\
       \hline
Eye
&    0.606 & 0.742 & 0.765 & 0.585 & 0.548 & 0.544 & 0.420\\
       \hline
\end{tabular}
\caption{Average adjusted Rand index (ARI) for 100 simulation runs of
size $n=500$ of each distribution.} \label{tab:ari}
\end{table}

In view of Table \ref{tab:ari}, none of the methods compared is uniformly the best. In the group of
the mean shift procedures, the use of the PI bandwidth seems to exhibit the best overall
performance. The CV choice can be rated second best, with similar or even slightly (but not
significantly) better average ARI in some cases. The SCV bandwidth shows an unexpectedly inferior
performance for the normal mixture models, but it has an acceptable behaviour for the models with
non-standard cluster shapes. Finally, the normal scale rule NR is clearly inferior in four out of
the five models, but it performs surprisingly well for the broken ring model; since it is the least
intensive method in computational terms, it could be useful at least to provide a quick initial
analysis, especially in higher dimensions.

The comparison with the parametric method MCLUST followed the expected guidelines: for the normal
mixture models MCLUST showed good results, especially for the difficult quadrimodal density, but it
seems unable to adapt itself to the non-standard cluster shape situations. On the contrary, CLUES
is not very powerful for a standard setup with ellipsoidal clusters, but seems to performs
reasonably well for non-standard problems. Finally, PDFC shows remarkable results in the simulation
study, in spite of the ad hoc choice of the bandwidth in which it is based, and its performance is
comparable to that of the best mean shift procedure, with the only exception of the eye model.
Surely a more careful study of the bandwidth selection problem would improve the quality of the
PDFC method further.

\subsection{Real data examples}
\color{black}

The mean shift algorithm in conjunction with the new proposed bandwidth selection rules was also
applied to some real data sets. It is well-known that the kernel density estimator tends to produce
spurious bumps (i.e., unimportant modes caused by a single observation) in the tails of the
distribution, and that this problem seems enhanced in higher dimensions, due to the empty space
phenomenon and the curse of dimensionality \citep[see, for instance,][Chapter 4]{Sim96}. For real
data sets, this may result in a number of data points forming singleton clusters after applying the
mean shift algorithm.

Furthermore, in some applications the researcher may be interested in forming more homogeneous
groups so that, say, insignificant groups of size less than $\alpha\%$ of the biggest group are not
allowed in the outcome of the clustering algorithm. This goal can be achieved as follows: apply the
mean shift algorithm to the whole data set and identify all the data points forming groups of size
less than $\alpha\%$ of the biggest group, then leave those singular data points out of the
estimation process in the mean shift algorithm and re-compute the data-based bandwidth and the
density and density gradient estimators in (\ref{eq:meanshift}) using only non-singular data
points. Since the mean shift algorithm produces a partition of the whole space, these left-out data
points can be naturally assigned to any of the corresponding newly obtained clusters. If this new
assignment again contains insignificant clusters then iterate the process until the eventual
partition satisfies the desired requirements. This correction is similar (although a little
different) to the stage called ``merging clusters based on the coverage rate" in \cite{LRL07}, and
will be referred henceforth as {\em correction for insignificant groups}.

\subsubsection{E.coli data}

\color{black}

The {\it E.~coli} data set is provided by the UCI machine learning database repository
\citep{FA10}. The original data were contributed by Kenta Nakai at the Institute of Molecular and
Cellular Biology of Osaka University. The data represent seven features calculated from the amino
acid sequences of $n=336$ E.coli proteins, classified in eight classes according to their
localization sites, labeled imL (2 observations), omL (5), imS (2), om (20), pp (52), imU (35), im
(77), cp (143). A more detailed description of this data set can be found in \cite{HN96}. Since two
of the original seven features are binary variables, only the remaining five continuous variables
($d=5$), scaled to have unit variance, were retained for the cluster analysis.

 The number of groups identified by the mean shift procedure with correction for insignificant
groups (using $\alpha=5\%$ as a default) was 5 for PI and SCV bandwidths, which is the natural
choice if the insignificant clusters imL, omL and imS are merged into bigger groups. The mean shift
algorithm found 6 groups using the NR bandwidth and 7 with the CV bandwidth. Since in this example
the true cluster membership is available from the original data, it is also possible to compare the
performance of the methods using the ARI. The ARIs for these configurations were 0.63 (NR
bandwidth), 0.671 (CV), 0.667 (PI) and 0.559 (SCV). In contrast, CLUES and PDFC indicated a
severely underestimated number of groups in the data, namely 3 and 2, respectively, and whereas
CLUES obtains a remarkably high ARI anyway (0.697), the performance of PDFC is poor for this data
set in ARI terms (0.386). MCLUST also gives a reasonable answer, with 6 groups and an ARI of 0.642.
\color{black}

\subsubsection{Olive oil data}

These data were introduced in \cite{FALT83}, and consist of eight chemical measurements on $n=572$
olive oil samples from three regions of Italy. The three regions R1, R2 and R3 are further divided
into nine areas, with areas A1 (25 observations), A2 (56), A3 (206) and A4 (36) in region R1
(totalling 323 observations); areas A5 (65) and A6 (33) in region R2 (totalling 98); and areas A7
(50), A8 (50) and A9 (51) in region R3 (totalling 151). Detailed cluster analyses of this data set
are given in \cite{S03} and \cite{AT07}. Taking into account the compositional nature of these
data, they were transformed following the guidelines in the latter reference, first dealing with
the effect of rounding zeroes when the chemical measurement was below the instrument sensitivity
level and then applying the additive log-ratio transform to place the data in a 7-dimensional
Euclidean space \citep[see][for a recent monograph on compositional data]{PGB11}. Then, cluster
analysis was carried out over the first five principal components of the scaled Euclidean
variables.

The results of the analysis indicated that whereas some methods seemed to target the partition of
the data into major regions, others tried hard to discover the sub-structure of areas. This was
clearly recognized when the ARIs of the groupings were computed either with respect to one
classification or the other. Naturally, if a method produced a grouping which was accurate with
respect to major regions, it had lower ARI with respect to the division into areas.

CLUES, PDFC and the mean shift algorithm using the NR bandwidth clearly favoured grouping the data
into major categories. The PDFC method obtained a remarkable ARI of 0.841 by clustering the data
into 3 groups, whereas CLUES only found 2 groups resulting in an ARI of 0.680. Using the NR
bandwidth the mean shift algorithm achieved an ARI of 0.920 with respect to the true grouping into
major regions; it correctly identified all the data points in regions R1 and R2, although region R3
appeared divided into several subregions.

In contrast, MCLUST and the mean shift algorithm combined with all the more sophisticated bandwidth
selectors tended to produce groupings closer to the assignment into smaller areas. MCLUST showed
the existence of 8 groups and achieved an ARI of 0.739 with respect to the true distribution into
areas. The mean shift analyses with the CV, PI and SCV bandwidths all found 7 groups, leading to
ARIs of 0.741 (CV bandwidth), 0.791 (PI) and 0.782 (SCV).

\color{black}

\section{Applications to bump-hunting with feature significance} \label{sec:feature}

It is not always easy to interpret visually estimates of multivariate derivatives. To assist us, we
use the significant negative curvature regions of \cite{DCKW08}, defined as the set containing the
values of $\bx\in\mathbb R^d$ such that the null hypothesis that the Hessian $\mathsf Hf(\bx)$ is
positive definite is significantly rejected. The appropriate kernel test statistic, null
distribution and adjustment for multiple testing is outlined in \cite{DCKW08} and implemented in
the \texttt{feature} library in \texttt{R}. Significant negative curvature regions corresponds to a
modal region in the density function, and hence a local maxima in data density. These authors
focused on the scale space approach of smoothing and so did not develop optimal bandwidth selectors
for their density derivative estimates.

Here, we compare the significant curvature regions obtained using a usual $r=0$ bandwidth selector
to those with an $r=2$  optimal bandwidth in Figure~\ref{fig:earthquake} on the earthquake data
from \cite{Sco92}. The recorded measurements are the latitude and longitude (in degrees) and depth
(in km) of epicenters of 510 earthquakes. Here, negative latitude indicates west of the
International Date Line, and negative depth indicates distances below the Earth's surface. The
depth is transformed using --log(--depth). For these transformed data, we use PI selectors
$\bH_{\PI,0}$ and $\bH_{\PI,2}$ and SCV selectors $\bH_{\SCV,0}$ and $\bH_{\SCV,2}$.

\enlargethispage{.5cm}

As expected from asymptotic theory, bandwidths for Hessian estimation are larger in magnitude than
bandwidths for density estimation. Moreover only the central modal region is present using
$\bH_{\PI,0}$, whereas with $\bH_{\PI,2}$, the three local modal regions are more clearly delimited
from the surrounding space, confirming the three modes obtained with subjective bandwidth selection
by \cite{Sco92}.

\begin{figure}[!htp]
\includegraphics[width=0.5\textwidth]{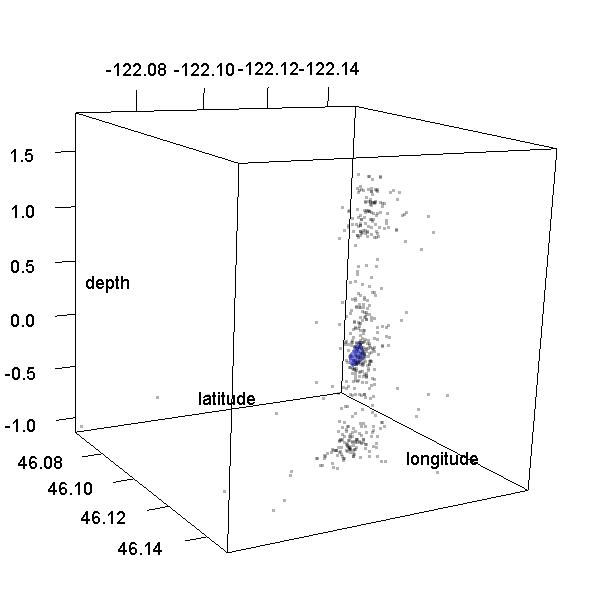}
\includegraphics[width=0.5\textwidth]{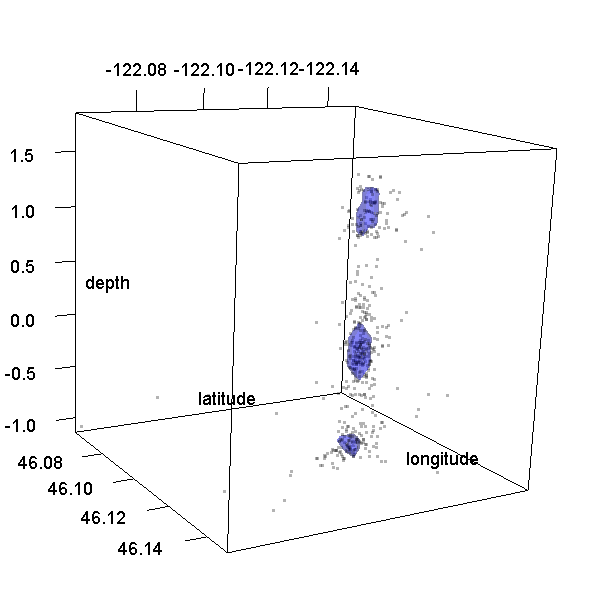} \\
\includegraphics[width=0.5\textwidth]{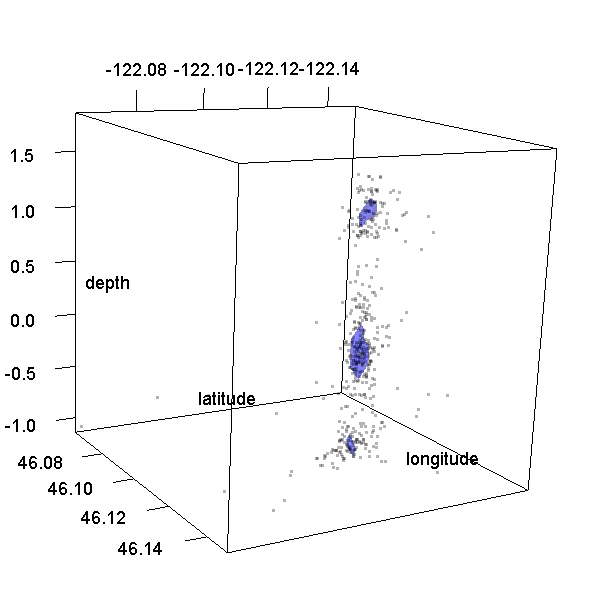}
\includegraphics[width=0.5\textwidth]{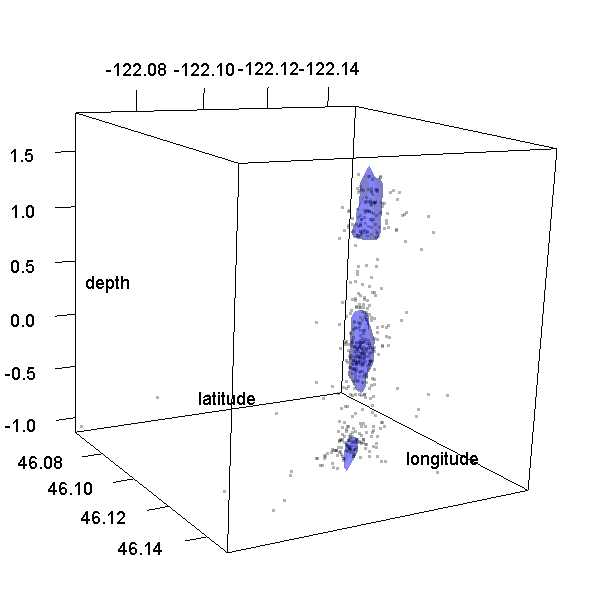} \\
\caption{Significant negative curvature regions (in blue). (Upper left) Plug-in selector $r=0$. (Upper right) Plug-in selector $r=2$. (Lower left) SCV selector $r=0$. (Lower right) SCV selector $r=2$. The significant curvature regions or modal regions are more clearly delimited
from the surrounding scatter point cloud with the selectors corresponding to second derivative.}
\label{fig:earthquake}
\end{figure}

\bigskip

\noindent{\bf Acknowledgments.} Grant MTM2010-16660 (both authors) from the Spanish Ministerio de Ciencia
e Innovaci\'on, and various fellowships (second author)
from the Institut Curie, France, and the Institute of Translational Sciences, France have supported this work.

\appendix
\section{Appendix: proofs}

Henceforth the following assumptions are made:
\begin{enumerate}
\item[(A1)] $K$ is a symmetric $d$-variate density such that $\int \bx \bx^\top K(\bx) \, d\bx=
    m_2(K) \bI_d$ and all its partial derivatives up to order $2r+1$ are bounded, continuous and
    square integrable.
\item[(A2)] $f$ is a density function with all its partial derivatives up to order $2r+6$
    bounded, continuous and square integrable.
\item[(A3)] $\bH = \bH_n$ is a sequence of bandwidth matrices such that all entries of $n^{-1}
    |\bH|^{-1/2} (\bH^{-1})^{\otimes r} $ and $\bH$ tend to zero as $n \rightarrow \infty$.
\end{enumerate}
These do not form a minimal set of assumptions, but they serve as useful starting point
for the results that we subsequently develop. Besides, in this section integrals without any integration limits are assumed to be integrated over the
appropriate Euclidean space. We also assume that suitable regularity conditions are satisfied so that
the exchange of term-by-term integration and differentiation of Taylor expansions are
well-defined.

\begin{proof}[Proof of Lemma~\ref{lem:asymHr}]
Reasoning as in Lemma~1 in \cite{DH05a}, it follows that $\vec(\hat\bH_r-\bH_{\MISE,r})$ is asymptotically
equivalent to
$-[\D_\bH^2\MISE_r(\bH_{\MISE,r})]^{-1}\D_\bH[\widehat{\mathrm{MISE}}_r-\mathrm{MISE}_r](\bH_{\MISE,r})$,
where $\D_\bH^2=\partial^2/(\partial\vec\bH\partial\vec^\top\bH)$ denotes the Hessian operator
corresponding to $\D_\bH$. Therefore, it suffices to show that $\D_\bH^2\MISE_r(\bH_{\MISE,r})= O(\mat{J}_{d^2})$. But for any $\bH$ with entries of order $O(n^{-2/(d+2r+4)})$, as $\bH_{\MISE,r}$, the results in \cite{CDW11} imply that $\MISE_r(\bH)$ is equivalent to $\AMISE_r(\bH)$ and, moreover, the smoothness assumptions ensure that
$\D_\bH^2\MISE_r(\bH)$ is of the same order as $\D_\bH^2\AMISE_r(\bH)$. And it is not hard to show that for the asymptotic integrated squared bias term we have $\D^2_\bH\big\{
\bpsi_{2r+4}^\top\big(\vec \bI_{d^r} \otimes(\VEC \bH)^{\otimes 2}\big)\big\}  = O(\mat{J}_{d^2})$ and similarly for the asymptotic integrated variance, thus finishing the proof.
\end{proof}

\subsection{Convergence rate for the CV bandwidth}

Lemma~\ref{lem:asymHr} shows that $\vec(\hat\bH_{\CV,r}-\bH_{\MISE,r})$ is asymptotically equivalent to
$\D_\bH[\CV_r-\MISE_r](\bH_{\MISE,r})$. Since $\mathbb E[{\rm CV}_r(\bH)]=\MISE_r(\bH)-\tr\mat R(\D^{\otimes r}f)$ for all $\bH$, it follows that the order of
$\vec(\hat\bH_{\CV,r}-\bH_{\MISE,r})$ is given by the (root) order of
\begin{align*}
\Var\big\{\D_\bH&[\mathrm{CV}_r-\mathrm{MISE}_r](\bH_{\MISE,r})\big\}\\
&\sim\Var\bigg\{[n(n-1)]^{-1}\sum_{i\neq j}^n\D_\bH\big[\vec^\top(\bH^{-1})^{\otimes r}(\D^{\otimes 2r}\tilde K)_\bH(\bX_i-\bX_j)\big]\Big\vert_{\bH=\bH_{\MISE,r}}\bigg\},
\end{align*}
where $\tilde K=K*K-2K$. So denoting $\boldsymbol{\varphi}_\bH(\bx)=\D_\bH\big[\vec^\top(\bH^{-1})^{\otimes r}(\D^{\otimes 2r} \tilde K)_\bH(\bx)\big]$, by standard $U$-statistics theory the previous
variance is of the same order as $4n^{-1}(\mat\Xi_1-\mat\Xi_0)+2n^{-2}\mat\Xi_2,$ where
\begin{align*}
\mat\Xi_1&=\E[\boldsymbol{\varphi}_\bH(\bX_1-\bX_2)\boldsymbol{\varphi}_\bH(\bX_1-\bX_3)^\top]\\
\mat\Xi_2&=\E[\boldsymbol{\varphi}_\bH(\bX_1-\bX_2)\boldsymbol{\varphi}_\bH(\bX_1-\bX_2)^\top]\\
\mat\Xi_0&=\E[\boldsymbol{\varphi}_\bH(\bX_1-\bX_2)]\E[\boldsymbol{\varphi}_\bH(\bX_1-\bX_2)]^\top
\end{align*}
with $\bH$ of the order of $\bH_{\MISE,r}$, namely having all its entries of order $O(n^{-2/(d+2r+4)})$. The following lemma provides an explicit expression for the function $\boldsymbol{\varphi}_\bH(\bx)$ that will be helpful to evaluate $\mat\Xi_p, p=0,1,2$.

\begin{lemma}
\label{lem:varphi}
The function $\boldsymbol{\varphi}_\bH(\bx)$ can be explicitly expressed as
$\boldsymbol{\varphi}_\bH(\bx)=\mat A(\D^{\otimes 2r} \tilde K)_\bH(\bx)+\mat B\boldsymbol{\rho}_\bH(\bx)$
where the function $\boldsymbol{\rho}\colon\mathbb R^d\to\mathbb R^{d^{2r+2}}$ is given by $\boldsymbol{\rho}(\bx)=(\bI_{d^{2r}}\otimes\bx\otimes \bI_d)\D^{\otimes(2r+1)}\tilde K(\bx)$
and the matrices $\mat A\equiv\mat A(\bH)\in\mathcal M_{d^2\times d^{2r}}$, $\mat B\equiv\mat B(\bH)\in\mathcal M_{d^2\times d^{2r+2}}$ are defined as
\begin{align*}
\mat A&=-\tfrac12(\vec^\top\bH^{\otimes-r}\otimes\vec \bH^{-1})-r(\vec^\top\bH^{\otimes-(r-1)}\otimes\bH^{\otimes -2})\\
\mat B&=-[\vec^\top\bH^{\otimes -r}\otimes(\bH^{1/2}\otimes\bH^{1/2}+\bI_d\otimes\bH)^{-1}]
\end{align*}
where we understand that $\bH^{\otimes-r}=(\bH^{-1})^{\otimes r}=(\bH^{\otimes r})^{-1}$.
\end{lemma}

\begin{proof}
Since $\vect\bH^{\otimes-r}(\D^{\otimes 2r} \tilde K)_\bH(\bx)=|\bH|^{-1/2}\vec^\top\bH^{\otimes -r}\D^{\otimes 2r}\tilde K(\bH^{-1/2}\bx)$, its differential is decomposed into three terms
\begin{align*}
d\big(\vect\bH^{\otimes-r}(\D^{\otimes 2r} \tilde K)_\bH(\bx)\big)
&= d(|\bH|^{-1/2}) \vect \bH^{\otimes -r} \D^{\otimes 2r} \tilde K (\bH^{-1/2} \bx) \\
&\quad + |\bH|^{-1/2} (d\vec \bH^{\otimes -r})^\top \D^{\otimes 2r} \tilde K (\bH^{-1/2} \bx) \\
&\quad + |\bH|^{-1/2} \vect \bH^{\otimes -r} d\big(\D^{\otimes 2r} \tilde K (\bH^{-1/2} \bx)\big).
\end{align*}
From \cite{CD10}, the differentials involved in the first two terms can be expressed as
\begin{align*}
d(|\bH|^{-1/2}) &= - \tfrac{1}{2} |\bH|^{-1/2} (\vect \bH^{-1}) d \vec\bH\qquad\qquad\text{and}\\
d(\vec \bH^{\otimes -r}) &= -\mat\Gamma_r[\vec \bH^{\otimes -(r-1)} \otimes \bH^{\otimes -2}] d \vec\bH,
\end{align*}
where $\mat\Gamma_r$ is a matrix such that $\mat\Gamma_r^\top\D^{\otimes 2r}=r\D^{\otimes 2r}$.
For the third term,
\begin{align*}
\vect \bH^{\otimes -r} d\big(\D^{\otimes 2r} \tilde K (\bH^{-1/2} \bx)\big)
&= \vect \bH^{\otimes -r} \big((\D^{\otimes 2r} \D^\top) \tilde K\big) (\bH^{-1/2} \bx) d(\bH^{-1/2}\bx) \\
&= \D^{\otimes 2r+1} \tilde K(\bH^{-1/2} \bx)^\top ( \vec \bH^{\otimes -r} \otimes \bI_d) d(\bH^{-1/2}\bx),
\end{align*}
since $[\D(\D^\top)^{\otimes 2r}]\vec\bH^{\otimes -r}=\vec\big(\bI_d[\D(\D^\top)^{\otimes 2r}]\vec\bH^{\otimes -r}\big)=(\vect\bH^{\otimes-r}\otimes\bI_d)\D^{\otimes 2r+1}$. Finally, using $d \vec\bH^{-1/2} = -(\bH^{1/2} \otimes \bH + \bH \otimes \bH^{1/2})^{-1} d\vec\bH$ from \cite{CD10}, it follows that $d( \bH^{-1/2} \bx) = (\bx^\top \otimes \bI_d) d \vec\bH^{-1/2} = -(\bx^\top\bH^{-1/2}
\otimes \bI_d)(\bI_d \otimes \bH+\bH^{1/2} \otimes \bH^{1/2})^{-1} d\vec\bH $.
Thus the derivative reads
\begin{align*}
\D_\bH\big(\vect\bH^{\otimes-r}(\D^{\otimes 2r} \tilde K)_\bH(\bx)\big)
&=-\tfrac{1}{2} |\bH|^{-1/2} (\vect \bH^{\otimes -r} \otimes \vec \bH^{-1}) \D^{\otimes 2r} \tilde K (\bH^{-1/2} \bx)   \\
&\quad - r|\bH|^{-1/2} (\vect \bH^{\otimes -(r-1)} \otimes \bH^{\otimes -2}) \D^{\otimes 2r} \tilde K (\bH^{-1/2} \bx) \\
&\quad -|\bH|^{-1/2} (\bH^{1/2} \otimes \bH^{1/2} + \bI_d \otimes \bH)^{-1} (\bH^{-1/2}\bx \otimes \bI_d)   \\
&\quad \times (\vect \bH^{\otimes -r}  \otimes \bI_d ) \D^{\otimes 2r+1} \tilde K (\bH^{-1/2} \bx).
\end{align*}
The central factors of the third term on the right hand side can be rewritten as
\begin{align*}
(\bH^{1/2} \otimes \bH^{1/2} + \bI_d \otimes &\bH)^{-1} (\bH^{-1/2}\bx \otimes \bI_d)
(\vect \bH^{\otimes -r}  \otimes \bI_d )\\
&= (\bH^{1/2} \otimes \bH^{1/2} + \bI_d \otimes \bH)^{-1}
(\vect \bH^{\otimes -r}  \otimes \bH^{-1/2}\bx \otimes \bI_d)\\
&= \vect \bH^{\otimes -r} \otimes  [(\bH^{1/2} \otimes \bH^{1/2} + \bI_d \otimes \bH)^{-1}(\bH^{-1/2}\bx \otimes \bI_d)]\\
&= [\vect \bH^{\otimes -r} \otimes  (\bH^{1/2} \otimes \bH^{1/2} + \bI_d \otimes \bH)^{-1}] (\bI_{d^{2r}} \otimes \bH^{-1/2}\bx \otimes \bI_d),
\end{align*}
as desired.
\end{proof}

We now return to the task of finding the asymptotic order of $\mat\Xi_0$, $\mat\Xi_1$ and $\mat\Xi_2$. For that, some preliminary notation is needed. For any real function $a$ we denote its vector moment of order $p$ as $\bmu_p(a)=\int_{\mathbb R^d}\bx^{\otimes p}a(\bx)d\bx$. For instance, \cite{CD11} showed that $\mu_0(\tilde K)=-1$, $\bmu_1(\tilde K)=\bmu_2(\tilde K)=\bmu_3(\tilde K)=0$ and $\bmu_4(\tilde K)=6 m_2(K)^2
\boldsymbol{\mathcal S}_{d,4}(\VEC \bI_d)^{\otimes 2}$, where $\boldsymbol{\mathcal S}_{d,r}$ denotes the symmetrizer matrix of order $r$ \citep[see][]{Hol85}, defined as the (only) matrix such that pre-multiplying a Kronecker product of any $r$ vectors in $\mathbb R^d$ by $\S_{d,r}$ results in the average of all possible permutations of the $r$-fold product. We also introduce here the notation $\mat K_{m,n}$ for the commutation matrix of order $mn\times mn$ \citep{MN79}.

So taking this into account, for the calculation of the asymptotic order of $\mat\Xi_0$, a fourth order Taylor expansion of $\D^{\otimes 2r}f(\bx-\bH^{1/2}\bz)$, in the form of \citet[Theorem 1.4.8]{KvR05} or \cite{CDW11}, gives
\begin{align*}
(\D^{\otimes 2r}\tilde K&)_\bH*f(\bx)\\&=\int\D^{\otimes 2r}\tilde K(\bz)f(\bx-\bH^{1/2}\bz)d\bz\\
&=(\bH^{1/2})^{\otimes 2r}\int \tilde K(\bz)\D^{\otimes 2r}f(\bx-\bH^{1/2}\bz)d\bz\\
&\sim(\bH^{1/2})^{\otimes 2r}\sum_{p=0}^4\frac{(-1)^p}{p!}\int \tilde K(\bz)\big[\bI_{d^{2r}}\otimes(\bz^\top \bH^{1/2})^{\otimes p}\big]\D^{\otimes 2r+p}f(\bx)d\bz\\
&=(\bH^{1/2})^{\otimes 2r}\sum_{p=0}^4\frac{(-1)^p}{p!}\big[\bI_{d^{2r}}\otimes(\bmu_p(\tilde K)^\top (\bH^{1/2})^{\otimes p})\big]\D^{\otimes 2r+p}f(\bx)\\
&=-(\bH^{1/2})^{\otimes 2r}\D^{\otimes 2r}f(\bx)+\tfrac{1}{4}m_2(K)^2(\bH^{1/2})^{\otimes 2r}\big[\bI_{d^{2r}}\otimes((\VEC^\top  \bH)^{\otimes 2}\boldsymbol{\mathcal S}_{d,4})\big]\D^{\otimes 2r+4}f(\bx)\\
&=-(\bH^{1/2})^{\otimes 2r}\D^{\otimes 2r}f(\bx)+\tfrac{1}{4}m_2(K)^2\big[(\bH^{1/2})^{\otimes 2r}\otimes(\VEC^\top  \bH)^{\otimes 2}\big]\D^{\otimes 2r+4}f(\bx),
\end{align*}
Therefore, since $\vec^\top \bH^{\otimes -r}(\bH^{1/2})^{\otimes 2r}=\vec^\top \bI_{d^r}$ and $\D_\bH(\vec\bH)^{\otimes2}=(\bI_{d^2}\otimes\vec^\top \bH)(\bI_{d^4}+\mat K_{d^2,d^2})$, we obtain
\begin{align*}
\boldsymbol\varphi_\bH*f(\bx)&=\D_\bH\big[\vec^\top \bH^{\otimes -r}(\D^{\otimes 2r}\tilde K)_\bH*f(\bx)\big]\nonumber\\
&\sim\D_\bH\big\{\tfrac{1}{4}m_2(K)^2\big[\vec^\top \bI_{d^r}\otimes(\VEC^\top  \bH)^{\otimes 2}\big]\D^{\otimes 2r+4}f(\bx)\big\}\nonumber\\
&=\tfrac{1}{4}m_2(K)^2\big(\vec^\top \bI_{d^r}\otimes\bI_{d^2}\otimes\vec^\top \bH\big)[\bI_{d^{2r}}\otimes(\bI_{d^2}+\mat K_{d^2,d^2})]\D^{\otimes 2r+4}f(\bx)\nonumber\\
&=\tfrac{1}{2}m_2(K)^2\big(\vec^\top \bI_{d^r}\otimes\bI_{d^2}\otimes\vec^\top \bH\big)\D^{\otimes 2r+4}f(\bx).\label{phiHf}
\end{align*}
Using this,
\begin{align*}
\E[\boldsymbol{\varphi}_\bH(\bX_1-\bX_2)]&=\int\boldsymbol\varphi_\bH*f(\bx)f(\bx)d\bx\sim\tfrac{1}{2}m_2(K)^2\big(\vec^\top \bI_{d^r}\otimes\bI_{d^2}\otimes\vec^\top \bH\big)\boldsymbol\psi_{2r+4}
\end{align*}
and $\mat\Xi_0=O(\mat J_{d^2})\vec\bH\vec^\top \bH.$

Similarly,
\begin{align*}
\mat\Xi_1&=\int\boldsymbol{\varphi}_\bH*f(\bx)\boldsymbol{\varphi}_\bH*f(\bx)^\top f(\bx)d\bx\\
&\sim\tfrac{1}{4}m_2(K)^4\big(\vec^\top\bI_{d^r}\otimes\bI_{d^2}\otimes\vec^\top\bH\big)\bigg\{\int\D^{\otimes 2r+4}f(\bx)\D^{\otimes 2r+4}f(\bx)^\top f(\bx)d\bx\bigg\}\\
&\quad\times
\big(\vec\bI_{d^r}\otimes\bI_{d^2}\otimes\vec\bH\big)\\
&=O(\mat J_{d^2})\vec\bH\vec^\top\bH.
\end{align*}
Finally, note that $\mat\Xi_2=|\bH|^{-1/2}\mathbb
E[(\boldsymbol\varphi\boldsymbol\varphi^\top)_\bH(\bX_1-\bX_2)]$, where
$\boldsymbol\varphi(\bx)=\mat A(\D^{\otimes 2r}\tilde K)(\bx)+\mat B\boldsymbol\rho(\bx)$, which
also depends on $\bH$ through $\mat A$ and $\mat B$. Besides, $$\mathbb
E[(\boldsymbol\varphi\boldsymbol\varphi^\top)_\bH(\bX_1-\bX_2)]=\iint(\boldsymbol\varphi\boldsymbol\varphi^\top)(\bz)f(\by)f(\by+\bH^{1/2}\bz)
d\by d\bz\sim R(f)\int\boldsymbol\varphi(\bz)\boldsymbol\varphi(\bz)^\top d\bz$$ which, in view of
Lemma~\ref{lem:varphi}, leads to $\mat\Xi_2=O(\mat
J_{d^2}|\bH|^{-1/2})\vec\bH^{\otimes-(r+1)}\vec^\top\bH^{\otimes-(r+1)}$.

Putting all these together, since every element of $\bH_{\MISE,r}$ is $O(n^{-2/(d+2r+4)})$,
$$4n^{-1}(\mat\Xi_1-\mat\Xi_0)+2n^{-2}\mat\Xi_2\sim O(\mat J_{d^2}n^{-d/(d+2r+4)})\vec\bH_{\MISE,r}\vec^\top\bH_{\MISE,r}$$
and therefore
$\vec(\hat\bH_{\CV,r}-\bH_{\MISE,r})=O(\mat J_{d^2}n^{-d/(2d+4r+8)})\vec\bH_{\MISE,r}$.

\subsection{Convergence rate for the PI bandwidth}

Henceforth, in addition to (A1)--(A3) the following assumptions on the pilot kernel $L$ and the pilot bandwidth $\bG$ are made:
\begin{enumerate}
\item[(A4)] $L$ is a symmetric $d$-variate density such that $\int \bx \bx^\top L(\bx) \, d\bx=
    m_2(L) \bI_d$ and all its partial derivatives up to order $2r+4$ are bounded, continuous and square
    integrable.
 \item[(A5)] $\bG = \bG_n$ is a sequence of bandwidth matrices such that all entries of
$n^{-1} |\bG|^{-1/2} (\bG^{-1})^{\otimes r+2}$ and $\bG$ tend to zero as $n \rightarrow \infty$.
\end{enumerate}

To make use of Lemma \ref{lem:asymHr} once more, notice that the difference between the MISE and its estimate is
\begin{align*}
\mathrm{PI}_r(\bH)-\mathrm{MISE}_r(\bH)\sim(-1)^r \tfrac{m_2(K)^2}{4}
(\hat{\boldsymbol{\psi}}_{2r+4}(\bG) - \boldsymbol{\psi}_{2r+4})^\top\big(\vec \bI_{d^r}\otimes(\vec\bH)^{\otimes2}\big)
\end{align*}
so taking into account $\D_\bH(\vec\bH)^{\otimes2}=(\bI_{d^2}\otimes\vec^\top \bH)(\bI_{d^4}+\mat K_{d^2,d^2})$ again, we come to
\begin{align*}
\D_\bH[\mathrm{PI}_r(\bH)-\mathrm{MISE}_r(\bH)]&\sim(-1)^r \tfrac{m_2(K)^2}{2}(\VEC^\top \bI_{d^r}\otimes \bI_{d^2} \otimes \VEC^\top \bH)(\hat{\boldsymbol{\psi}}_{2r+4}(\bG) - \boldsymbol{\psi}_{2r+4}),
\end{align*}
so that the performance of $\bH_{\PI,r}$ is determined by the performance of
$\hat{\boldsymbol{\psi}}_{2r+4}(\bG)$ as an estimator of $\boldsymbol{\psi}_{2r+4}$.

From Theorem 2 in \citet{CD10} the optimal pilot bandwidth $\bG$ for the estimator $\hat{\boldsymbol{\psi}}_{2r+4}(\bG)$ is of order $n^{-2/(d+2r+6)}$, leading to $\mathbb E\big[ \|\hat{\boldsymbol{\psi}}_{2r+4}(\bG) - \boldsymbol{\psi}_{2r+4}
\|^2\big]=O(n^{-4/(d+2r+6)})$, and then $\D_\bH[\mathrm{PI}_r(\bH)-\mathrm{MISE}_r(\bH)]=O_P(n^{-2/(d+2r+6)}\mat{J}_{d^2})\VEC \bH.$
So finally we arrive to $\vec(\hat\bH_{\PI,r}-\bH_{\MISE,r})=O_P(n^{-2/(d+2r+6)}\mat{J}_{d^2})\VEC \bH_\MISE$
by applying Lemma~\ref{lem:asymHr}.

\subsection{Convergence rate for the SCV bandwidth}

As in \cite{CD11}, it can be shown that the function $\MISE_r$ can be replaced for $\MISE2_r$ everywhere in the asymptotic analysis, since the difference between their respective minimizers is of relative order faster than $n^{-1/2}$, which is the fastest attainable rate in bandwidth selection \citep{HM91}.

So to apply Lemma \ref{lem:asymHr} it is also possible consider $\MISE2_r$ instead of $\MISE_r$, hence we focus on analyzing the difference
$\SCV_r(\bH)-\MISE2_r (\bH)$ at $\bH$ of the same order as $\bH_{\MISE,r}$. To begin with, note that using a fourth order Taylor expansion of
$\D^{\otimes 2r}\bar{L} (\bG^{-1/2}\bx - \bG^{-1/2} \bH^{1/2}\bz)$ results in
\begin{align*}
\bar\Delta_\bH*&\D ^{\otimes 2r}\bar L_\bG (\bx)\\ &=
\int \bar{\Delta}_\bH (\bz) \D^{\otimes 2r} \bar{L}_\bG (\bx -\bz) \, d\bz \\
&= |\bG|^{-1/2} (\bG^{-1/2})^{\otimes 2r} \int \bar{\Delta} (\bz) \D^{\otimes 2r}\bar{L} (\bG^{-1/2}\bx - \bG^{-1/2} \bH^{1/2}\bz) \, d\bz \\
&\sim |\bG|^{-1/2} (\bG^{-1/2})^{\otimes 2r} \sum_{p=0}^4 \frac{(-1)^p}{p!} \int \bar{\Delta} (\bz) [\bI_{d^{2r}} \otimes
(\bz^\top \bH^{1/2} \bG^{-1/2})^{\otimes p}]  \D^{\otimes 2r+p} \bar{L}(\bG^{-1/2}\bx) \, d\bz\\
&= \tfrac{1}{4} m_2(K)^2 |\bG|^{-1/2} (\bG^{-1/2})^{\otimes 2r} [\bI_{d^{2r}} \otimes (\vect \bH)^{\otimes 2} (\bG^{-1/2})^{\otimes 4} ]\D^{\otimes 2r+4} \bar{L}(\bG^{-1/2}\bx)\\
&= \tfrac{1}{4} m_2(K)^2|\bG|^{-1/2} [\bI_{d^{2r}} \otimes (\vect \bH)^{\otimes 2}] (\bG^{-1/2})^{\otimes (2r+4)} \D^{ \otimes 2r+4} \bar{L}(\bG^{-1/2}\bx)\\
&= \tfrac{1}{4} m_2(K)^2 [\bI_{d^{2r}} \otimes (\vect \bH)^{\otimes 2}]  \D^{ \otimes 2r+4} \bar{L}_\bG(\bx),
\end{align*}
where we have made use of the fact that $\mu_0(\bar \Delta)=\bmu_1(\bar\Delta)=\bmu_2(\bar\Delta)=\bmu_3(\bar\Delta)=0$ and $\bmu_4(\bar \Delta)= 6 m_2(K)^2
\boldsymbol{\mathcal S}_{d,4}(\VEC \bI_d)^{\otimes 2}$, and that the entries of $\bG^{-1}\bH$ tend to zero as a consequence of (A3) and (A5).

This asymptotic approximation is then used to expand the terms in
\begin{multline*}
\E [ \SCV_r (\bH) - \MISE2_r (\bH)]= (-1)^r \vect \bI_{d^r} \bigg\{ n^{-1} \bar{\Delta}_\bH *\D^{\otimes 2r} \bar{L}_\bG(0) \\
+ (1-n^{-1})\E\big[  (\bar{\Delta}_\bH * \D^{\otimes 2r} \bar{L}_\bG)(\bX_1 - \bX_2)\big] - \int \bar{\Delta}_\bH * \D^{\otimes 2r} f(\bx)
f(\bx) \, d\bx\bigg\}.
\end{multline*}
Precisely, for the first term we have
\begin{align*}
\bar{\Delta}_\bH *\D^{\otimes 2r} \bar{L}_\bG(0) &\sim \tfrac{1}{4} m_2(K)^2 |\bG|^{-1/2} [\bI_{d^{2r}} \otimes (\vect \bH)^{\otimes 2}](\bG^{-1/2})^{\otimes 2r+4} \D^{\otimes 2r+4} \bar{L}(0),
\end{align*}
and for the second term
\begin{align*}
\E &\big[(\bar{\Delta}_\bH * \D^{\otimes 2r} \bar{L}_\bG)(\bX_1 - \bX_2)\big] \\
&\sim \tfrac{1}{4}m_2(K)^2[\bI_{d^{2r}} \otimes (\vect \bH)^{\otimes 2}]
\iint \D^{\otimes 2r+4} \bar{L}_\bG (\bx - \by) f(\bx) f(\by) \, d\bx d\by  \\
&= \tfrac{1}{4}m_2(K)^2 [\bI_{d^{2r}} \otimes (\vect \bH)^{\otimes 2}]
\iint  \bar{L}_\bG (\bx - \by) \D^{\otimes 2r+4}f(\bx) f(\by) \, d\bx d\by  \\
&\sim \tfrac{1}{4}m_2(K)^2 [\bI_{d^{2r}} \otimes (\vect \bH)^{\otimes 2}]
\iint  \bar{L} (\bw) \sum_{p=0}^2 \frac{(-1)^p}{p!}[\bI_{d^{2r+4}} \otimes (\bw^\top \bG^{1/2})^{\otimes p}] \\
&\quad \times \D^{\otimes 2r+4+p} f(\by) f(\by) \, d\bw d\by  \\
&= \tfrac{1}{4}m_2(K)^2 [\bI_{d^{2r}} \otimes (\vect \bH)^{\otimes 2}]
\sum_{p=0}^2  \frac{(-1)^p}{p!}[\bI_{d^{2r+4}} \otimes \{\bmu_p(\bar{L})^\top (\bG^{1/2})^{\otimes p}\}]\bpsi_{2r+4+p}\\
&=\tfrac{1}{4}m_2(K)^2 [\bI_{d^{2r}} \otimes (\vect \bH)^{\otimes 2}]\bpsi_{2r+4}
+ \tfrac{1}{4}m_2(K)^2 m_2(L)[\bI_{d^{2r}} \otimes (\vect \bH)^{\otimes 2} \otimes \vec^\top \bG] \bpsi_{2r+6},
\end{align*}
since $\bmu_0(\bar{L}) = 1, \bmu_1(\bar{L}) = 0 $ and $\bmu_2(\bar{L}) = 2 \bmu_2(L) = 2m_2(L) \vec \bI_d$.
Finally, noting that $\D^{\otimes 2r}\tilde K_\bH=(\bH^{-1/2})^{\otimes 2r}(\D^{\otimes 2r}\tilde K)_\bH$ and making use of the previously obtained expansion for $(\D^{\otimes 2r}\tilde K)_\bH*f$, the third term is
\begin{align*}
\int \bar{\Delta}_\bH * \D^{\otimes 2r} f(\bx) f(\bx) \, d\bx &= \int \D^{\otimes2r}\tilde K_\bH*f(\bx) f(\bx)\, d\bx+\bpsi_{2r}\\
&\sim \tfrac{1}{4}m_2(K)^2\big[\bI_{d^{2r}}\otimes(\VEC^\top \bH)^{\otimes 2}\big] \bpsi_{2r+4}.
\end{align*}
Thus,
\begin{align*}
\E [\SCV_r (\bH)- \MISE2_r (\bH)] &\sim
\tfrac{1}{4} m_2(K)^2 n^{-1}|\bG|^{-1/2}[\vec^\top\bI_{d^r} \otimes (\vect \bH)^{\otimes 2}](\bG^{-1/2})^{\otimes 2r+4} \D^{\otimes 2r+4} \bar{L}(0) \\
&\quad +  \tfrac{1}{4}m_2(K)^2 m_2(L)[\vec^\top\bI_{d^r} \otimes (\vect \bH)^{\otimes 2} \otimes \vec^\top \bG] \bpsi_{2r+6}
\end{align*}
Calculations in Section~\ref{sec:3} give $\bG$ is order $n^{-2/(2r+d+6)}$, as for the plug-in selector, so substituting to this into
the derivative of the previous equation yields
\begin{align*}
\E\{\D_\bH [\SCV_r (\bH)- \MISE2_r (\bH)] \}& = O ([n^{-1} |\bG|^{-1/2}(\tr \bG)^{-r-2} + \tr \bG] \mat{J}_{d^2}) \vec \bH \\
&=O(n^{-2/(2r+d+6)} \mat{J}_{d^2}) \vec \bH.
\end{align*}
Lemma~\ref{lem:asymHr} shows that
$\vec(\hat\bH_{\SCV,r}-\bH_{\MISE,r})$ is asymptotically equivalent to $\D_\bH[\mathrm{SCV}_r-\mathrm{MISE2}_r](\bH_{\MISE,r})$.
Since it was stated in Section~\ref{sec:3} that $\E \big[\| \vec (\hat{\bH}_{\SCV,r} - \bH_{\MISE,r})\|^2\big]$ is dominated
by its squared bias term, then
$\vec(\hat\bH_{\SCV,r}-\bH_{\MISE,r}) = O_P(n^{-2/(2r+d+6)} \mat{J}_{d^2}) \vec \bH_{\MISE,r}$.

\bibliographystyle{apalike}

\end{document}